\documentclass[10pt]{amsart}

\usepackage{amsfonts}
\usepackage{textcomp}
\usepackage{amsthm}
\usepackage{amssymb}
\usepackage{xcolor}
\usepackage{graphicx}

\title[Decoding Rauzy Induction Effectively]{Decoding Rauzy Induction: An Effective Answer to Bufetov's Question}
\author[J. Fickenscher]{Jon Fickenscher}
\email{jonfick@princeton.edu}
\address{Fine Hall, Washington Road, Princeton, NJ, 08544, USA.}
\date{\today}
\dedicatory{This paper is dedicated to my wonderful advisor and mentor, William A.\ Veech.}

\newtheorem{theo}{Theorem}[section]
\newtheorem{lemm}[theo]{Lemma}
\newtheorem{coro}[theo]{Corollary}
\newtheorem{prop}[theo]{Proposition}
\newtheorem{algo}{Algorithm}
\newtheorem{ques}{Question}
\newtheorem{main}{Main Theorem}
\newtheorem{defi}[theo]{Definition}
\newtheorem{rema}[theo]{Remark}
\newtheorem{exam}[theo]{Example}

\theoremstyle{definition}

\newcommand{\AAA}{\mathcal{A}}
\newcommand{\BBB}{\mathcal{B}}

\newcommand{\LLL}{\mathcal{L}}
\newcommand{\NNN}{\mathcal{N}}

\newcommand{\QQQ}{\mathcal{Q}}
\newcommand{\RRR}{\mathcal{R}}
\newcommand{\SSS}{\mathcal{S}}

\newcommand{\NN}{\mathbb{N}}
\newcommand{\ZZ}{\mathbb{Z}}

\newcommand{\INTi}[1]{\ensuremath{\mbox{\textup{\textlbrackdbl}} #1 \mbox{\textup{\textrbrackdbl}}}}
\newcommand{\INTio}[1]{\ensuremath{\mbox{\textup{\textlbrackdbl}} #1 \mbox{\,\textup{\textlbrackdbl}}}}

\newcommand{\PAIRS}[1]{\mathbf{irr}(#1)}

\newcommand{\PERM}[1]{\mathfrak{S}_{#1}}

\newcommand{\IRR}[1]{\PERM{#1}^{0}}

\newcommand{\ID}[1]{\mathrm{id}_{#1}}

\newcommand{\mtrx}[1]{\left(\begin{matrix} #1 \end{matrix}\right)}

\newcommand{\stepref}[1]{\texttt{\color{red}\normalfont [\ref{#1}]}}

\newcommand{\newstepref}[1]{\texttt{\color{purple}\normalfont $\langle$\ref{#1}$\rangle$}}

\numberwithin{equation}{section}

\renewcommand{\theequation}{\textbf{\color{blue}\arabic{section}.\arabic{equation}}}

\begin{document}

\begin{abstract}
	A typical interval exchange transformation has an infinite sequence of matrices
		associated to it by successive iterations of Rauzy induction.
	In 2010, W.\ A.\ Veech answered a question of A.\ Bufetov by showing
		that the interval exchange itself may be recovered
		from these matrices and must be unique up to topological conjugation.
	In this work, we will improve upon these results
		by providing an algorithm to determine the initial transformation from a sufficiently long finite subsequence of these
		matrices.
	We also show the defined length to be necessary by constructing finite sequences of Rauzy induction
		with multiple distinct (even up to conjugacy) initial transformations.
\end{abstract}

\maketitle


\section{Introduction}

An $n$-interval exchange transformation ($n$-IET) is an invertible map $T :[0,L) \to [0,L)$ that divides the interval
	$[0,L)$ into $n$ sub-intervals whose lengths are given by a length vector $\lambda = (\lambda_1,\dots,\lambda_n)$ and then places them back in a different order
	as prescribed by a permutation $\pi$.
For almost all\footnote{For every \emph{irreducible} $\pi$ (Definition \ref{DEF_PERM_IRR}) and Lebesgue almost every $\lambda$.} $n$-IETs
	Rauzy's induction \cite{cRauzy} results in a new $n$-IET $T'$ and is defined by the first return
	of $T$ on an appropriately chosen subinterval $[0,L')$.
More specifically $L'$ is $L$ minus the minimum length of the rightmost subinterval before and after application of $T$.
The  information $\pi'$ and $\lambda'$ that defines $T'$ may be directly
	computed from $\pi$ and $\lambda$ and the relationships depend on the \emph{type} of the induction,
	meaning whether the rightmost before or after the application of $T$ had smallest length.
There is an invertible matrix $A$, called the \emph{visitation matrix}, that counts how often each subintevral of $T'$
	visits the subintervals of $T$ before returning to $[0,L')$.

If we apply Rauzy's induction to $T$ infinitely many times, we arrive at a sequence of such matrices
	$$
		A_1, A_2, A_3, A_4, A_5, A_6,\dots
	$$
so that $A_i$ is derived from the $i^{th}$ step of induction.

Given this infinite sequence of matrices and the initial permutation $\pi$, we may recover the originial $n$-IET
	up to topological conjugacy \cite{cVeechConj}.
If this initial $n$-IET is uniquely ergodic, meaning Lebesgue is the only $T$-invariant measure up to scaling,
	then the initial IET is itself unique given the initial length $L$.

\begin{ques}[A. Bufetov]
	Given only the sequence of $A_i$'s, can an initial permutation be determined and is this permutation unique?
\end{ques}

In \cite{cVeechBUF}, W. A. Veech gave an affirmative answer, and this result was used in \cite{cBufetov} and then \cite{cBufetov+Solomyak}.
However, Veech's proof was by contradiction and did not directly compute $\pi$.

As a follow-up to his first question, A. Bufetov wondered if this result would hold with even less information.
More precisely, suppose we know the \emph{products}
	$$
		\underbrace{A_1A_2\cdots A_{m_1}}_{B_1}, \underbrace{A_{m_1+1} A_{m_1+2}\cdots A_{m_2}}_{B_2}, \dots
	$$
for some increasing sequence $1\leq m_1 < m_2 < m_3 < \dots$, without knowing the individual matrices within each product.
Such products arise naturally when we apply Zorich's acceleration of Rauzy induction  \cite{cZorich}
	 or by inducing on other intervals.

\begin{ques}[A. Bufetov]
	Given only a sequence of $B_k$'s, where each $B_i$ is a product of $A_i$'s,
		can an initial permutation still be determined and is it still unique?
\end{ques}

The author gave an affirmative answer to this more general question in \cite{cFick}.
However, just as in \cite{cVeechBUF} the proof cannot be used to find the initial permutation.

The main result of this paper provides an effective answer to the first question and
	to the second question when restricted to Zorich's induction.\footnote{We slightly relax the definition of Zorich's induction for this paper.
	 The details will be provided in Sections \ref{SSEC_ZORICH_DEF}--\ref{SSEC_ZORICH_DEF2}.}
Before stating the main result, we note that a finite sequence of matrices given by Rauzy or  Zorich induction
	has a form of length which we will call in this paper $C$-complete, where $C$ is a positive integer.
We will be more precise with the definition later, but we note here that for a
	prescribed value of $C$ we can find an $N$ such that
	the finite sequence $\widetilde{A}_1,\dots,\widetilde{A}_N$ of initial matrices
	is $C$-complete.
(Here, use $\widetilde{A}_i$ to represent a matrix obtained by Rauzy or Zorich induction.)
Most importantly, we may determine this value of $N$ given only the matrices.
	
	\begin{main}
		If matrices
			$\widetilde{A}_1$,$\widetilde{A}_2$,\ldots are visitation matrices
		associated to moves of Rauzy/Zorich induction on an i.d.o.c. $n$-IET,
		consider the first $N$ matrices that are $C$-complete, $C \geq \log_2(n+1) - 1$.
		Then there is a unique permutation $\pi$ that
			can define any $n$-IET whose first $N$ steps can be given by
				$\widetilde{A}_1,\dots,\widetilde{A}_N$.
		Moreover, there is an explicit algorithm to determine $\pi$ from these $N$ matrices.
	\end{main}
	
	This statement is more precisely stated after proper notation has been introduced as
		Theorem \ref{THM_ZPAIR} and Corollary \ref{COR_ZPERM}.
	The main tool in the proofs is Algorithm \ref{ALGO_ZORICH} which builds information
		by starting with the $N^{th}$ matrix and iterating backwards to the first matrix.
	This algorithm is also what makes this result effective, as its success yields the unique permutation.
	
	As our second result, we show that the length we need for our first theorem is sharp by
		constructing sequences of matrices given by Rauzy induction
		that allow for more than one initial permutation and are $C$-complete with
		$$
			C = \lfloor\log_2(n) - 1\rfloor = \lceil \log_2(n+1)-2\rceil
		$$
	Note that $\lceil \log_2(n+1) -1\rceil$ is the minimum integer $C$ such that uniqueness is ensured by the first theorem.

	\begin{main}\label{MAIN2}
		For each $n\geq 4$ there exist a sequence
			$A_1,\dots, A_N$ of matrices given by Rauzy induction that is
				$C$-complete, $C = \lfloor \log_2(n)\rfloor -1\rfloor$,
				and the initial permutation is not unique.
	\end{main}
	
	Main Theorem \ref{MAIN2} is restated as Theorem \ref{THM_NEED_C} after we establish more notation.
	We note that matrices given by Rauzy induction a priori provide more information than
		those provided by Zorich induction.

	\subsection{Outline of Paper}

	Aside from defining our inductions and corresponding matrices,
		we will not discuss the underlying dynamics of IETs.
	We will work primarily with the notation from \cite{cViana} to define
		IETs and induction.
	This notation defines an IET not by a permutation $\pi$ on $\{1,\dots,n\}$
		but by a pair of bijections $p_t: \AAA \to \{1,\dots, n\}$ for $t\in \{0,1\}$
		for a alphabet $\AAA$ of size $n$.
	The translation from $(p_0,p_1)$ to $\pi$ and vice-versa will be discussed.
	We note here that there is an inherent ambiguity in notation for pairs:
		if we have an infinite sequence of matrices for induction
		that can be obtained from an IET with initial pair $(p_0,p_1)$ then
		this sequence may also be obtained by an IET with initial pair $(p_1,p_0)$
		(which we call the \emph{inverse} of $(p_0,p_1)$).
	This ambiguity does not arise when we take induction on permutations.
	
	In Section \ref{SEC_DEF} we establish our standard definitions and notations.
	We define the matrices from Rauzy induction first for pairs in Section \ref{SSEC_RAUZY_DEF}
		and then for permutations in Section \ref{SSEC_RAUZY_DEF2}.
	We define here what it means to be $C$-complete.
	In \ref{SSEC_PERM+PAIRS} we discuss how to explicitly translate between pair and permutation notation, leading
		to Lemma \ref{LEM_LIFTING_UP} which makes these translations precise.
	In Sections \ref{SSEC_ZORICH_DEF} and \ref{SSEC_ZORICH_DEF2} we define our generalization of Zorich's
		induction, first for pairs and then for permutations.
	The Zorich analogue of Lemma \ref{LEM_LIFTING_UP} is given by Corollary \ref{COR_LIFTING_ZOR}.
	Section \ref{SSEC_BREAKUP} discusses how to modify a sequence of Zorich induction matrices
		by possibly replacing some matrices with products in order to prepare for our main algorithm.
		
	Section \ref{SEC_MAIN1} is dedicated to proving our first main theorem, which is stated correctly for
		pairs as Theorem \ref{THM_ZPAIR} and for permutations as Corollary \ref{COR_ZPERM}.
	To do this, we define in Section \ref{SSEC_POPAIRS} \emph{partially ordered pairs} which we use to encode
		what information we currently know about possible pairs we could obtain by induction.
	We then turn to Algorithm \ref{ALGO_ZORICH} in Section \ref{SSEC_ALGO_DEF}.
	In particular, for a move from (unknown) pair $(p_0^{(i)},p_1^{(i)})$ to $(p_0^{(i+1)},p_1^{(i+1)})$ by a step
		of induction described by $\widetilde{A}_i$, we show how to take what we know about the $(i+1)^{st}$ pair
			and $\widetilde{A}_i$ to learn more about the $i^{th}$ pair.
	Using this algorithm, we prove in Section \ref{SSEC_MAIN1_PROOF}
		that if we take the first $N$ matrices that are $C$-complete
		for $C$ large enough then Algorithm \ref{ALGO_ZORICH} will produce
		one unique answer (up to taking the inverse).
	We then apply our explicit translations via Corollary \ref{COR_LIFTING_ZOR}
		to prove our main result for permutations.
	
	We then in Section \ref{SEC_MAIN2} show that Theorem \ref{THM_ZPAIR} and Corollary \ref{COR_ZPERM}
		provide sharp lower bounds on the $C$-complete requirement.
	After addressing smaller alphabet sizes, this is stated as Theorem \ref{THM_NEED_C}.
	In this result, we construct finite sequences of matrices that are $C$-complete,
		where $C$ is just below the bound given by our first main result,
		such that two distinct pairs (meaning not the same up to inverses)
			exist that may start this sequence of induction.
	Two important constructions arise in the proof of Lemma \ref{LEM_CUTTING_UNKNOWNS} as they describe
		sequences of induction that help ensure completeness
		while minimizing information gained.
		
	\subsection{Acknowledgments}
	The author thanks A. Bufetov for his question and
		W.\ A.\ Veech for his answer.
	He also thanks these two along with S.\ Ferenczi as well as these two for their enduring encouragement.
	The author is thankful every day for Laine, Charlie and Amy.

\section{Definitions}\label{SEC_DEF}

Let $\ZZ = \{0, \pm 1,\pm 2,\dots\}$ denote the set of integers,
	$\NN = \{1,2,3,\dots\}$ denote the positive integers
	and $\NN_0 = \NN \cup \{0\}$ denote the set of non-negative integers.
For integers $a,b\in \ZZ$ we let $\INTi{a,b}$ denote the interval of integers $x\in \ZZ$ satisfying
	$a\leq x \leq b$
	and $\INTio{a,b}$ denote the integers $y\in \ZZ$ satisfying $a\leq y <b$.
As special cases: if $a = b$ then $\INTi{a,b}=\INTio{a,b+1}$ is the set containing the single element $a=b$,
	if $b<a$ then $\INTi{a,b} = \INTio{a,b+1}$ is the empty set $\emptyset$
	and $\INTio{a,\infty}$ will denote the set of all integers $x\in \ZZ$
	satisfying $x\geq a$.
For a set $\AAA$ we let $|\AAA| \in \NN_0\cup \{\infty\}$ denote its cardinality.

\subsection{Rauzy Induction on Pairs}\label{SSEC_RAUZY_DEF}

In this section we discuss Rauzy induction using
	the prevalent notation for interval exchanges as used in
	\cite{cViana}.
We refer the reader to that text for a more complete discussion.

\begin{defi}
	Let $\AAA$ be a finite alphabet of two or more symbols.
	A \emph{pair over $\AAA$} is a tuple $(p_0,p_1)$ of
		bijections $p_t:\AAA \to \INTi{1,|\AAA|}$.
	A pair $(p_0,p_1)$ over $\AAA$ is \emph{irreducible}
		if
		$$
			p_0^{-1}(\INTi{1,k}) \neq p_1^{-1}(\INTi{1,k})
				\mbox{ for all }
			k\in \INTi{1,|\AAA|-1},
		$$
	meaning in other words that the sets of the first $k$ symbols
		according to each ordering rule $p_0$ and $p_1$ cannot coincide
		unless $k = |\AAA|$.
	The set of all irreducible pairs over $\AAA$ will be denoted by $\PAIRS{\AAA}$.
\end{defi}

If $(p_0,p_1)\in \PAIRS{\AAA}$ and $(p_0',p_1')$ is the
	resulting pair after applying Rauzy induction of type $t\in \{0,1\}$
	then
\begin{equation}\label{EQ_RIND_PAIRS}
	p'_t = p_t
	\mbox{ and }
	p'_{1-t}(b) = \begin{cases}
		p_{1-t}(b), &  \mbox{if }p_{1-t}(b) \in\INTi{1, p_{1-t}(w)},\\
		p_{1-t}(b)+1, &  \mbox{if }p_{1-t}(b)\in \INTi{ p_{1-t}(w)+1,|\AAA|-1},\\
		p_{1-t}(w)+1, &  \mbox{if }b = \ell,\\
	\end{cases}
\end{equation}
where $\ell = p_{1-t}^{-1}(|\AAA|)$ is the loser and $w = p_t^{-1}(|\AAA|)$
is the winner.
Using this notation, the $\AAA\times\AAA$ matrix $\Theta$ associated to this move of Rauzy induction
is
\begin{equation}\label{EQ_THETA_PAIRS}
	\Theta(a,b) = \begin{cases}
				1, & \mbox{if }a=b,\\
				1, & \mbox{if }a=w\mbox{ and } b = \ell,\\
				0, & \mbox{otherwise.}
			\end{cases}
\end{equation}

\begin{rema}\label{REM_EITHER_TYPE}
	As this definition indicates, if $\Theta$ describes a move
	of Rauzy induction then we can determine the winner $w$ and loser $\ell$
	of the move as they label the row and column respectively
	of the only off-diagonal $1$ entry.
	However, the type of the move cannot be determined
		from this matrix;
	 if $\Theta$ is the matrix associated to the move
		from $(p_0,p_1)$ to $(p_0',p_1')$
		and is of type $t$, then it is also the matrix associated to the
		move from $(p_1,p_0)$ to $(p_1',p_0')$ which is of type $1-t$.
\end{rema}

	It is a useful exercise to show that fixing two of the following:
		\begin{enumerate}
			\item the initial pair $(p_0,p_1)\in \PAIRS{\AAA}$,
			\item the resulting pair $(p_0',p_1')\in \PAIRS{\AAA}$ or
			\item the winner/loser tuple $(w,\ell)$,
		\end{enumerate}
	will uniquely determine the third.
	Furthermore, we may replace (3) above with knowing the type.

\begin{defi}\label{DEF_PAIR_PATH}
	A \emph{Rauzy path} or \emph{path of Rauzy induction}
		is an ordered (finite or infinite) sequence of moves by Rauzy induction.
	The length $N\in \NN\cup\{\infty\}$ is the number of moves in the path
		and the pairs visited by the path are $(p_0^{(j)},p_1^{(j)})$, $j\in \INTio{1,N+2}$
			(where $N + 2 = \infty$ if $N = \infty$),
	so that for each $j$ the pair $(p_0^{(j+1)},p_1^{(j+1)})$
		is the result of applying the $j^{th}$ inductive move to
		the pair $(p_0^{(j)},p_1^{(j)})$.
	We say that the path \emph{starts at} $(p_0^{(1)},p_1^{(1)})$
	and refer to this as the \emph{initial pair}.
\end{defi}

By the observation preceding this definition, if we know any pair $(p_0^{(j)},p_1^{(j)})$
	along the path and we know the types or winner/loser tuples of each induction move
	then we actually know each pair along the path.
However, this paper is concerned with determining any/all of the pairs given only the
	winner/loser tuples as this is the only information described by the
	sequence of matrices $\Theta_j$ associated to each moves for $j\in \INTi{1,N}$.
The observation in Remark \ref{REM_EITHER_TYPE} tells us by induction that we can always
	switch the rows in all pairs along a path without altering
	the winner and loser of each individual move.
This obstruction to determining a unique initial pair (which is considered proven by the
	dsicussions above) is stated here for future reference.

\begin{prop}\label{PROP_PAIRS_BOTH_TYPES}
	If $\Theta_1,\dots,\Theta_N$, $N\in \NN\cup\{\infty\}$, describes
		a path of Rauzy induction over $\AAA$ starting at $(p_0,p_1)\in \PAIRS{\AAA}$
	then this path may also start at $(p_1,p_0)$.
\end{prop}

While introduced here, the motivation for the following terminology will be given in Remark \ref{REM_INVERSE_DEF}
	after we discuss the connection between pairs and permutations.

\begin{defi}\label{DEF_PAIR_INVERSE}
	If $(p_0,p_1)\in \PAIRS{\AAA}$ then its \emph{inverse} is $(p_1,p_0)\in \PAIRS{\AAA}$.
\end{defi}

So now Proposition \ref{PROP_PAIRS_BOTH_TYPES}
	tells us that we can at best uniquely determine the initial pair
	from a sequence of $\Theta_j$'s \emph{up to inverses}.
Before concluding this section, we will need to discuss a measure of length
	on paths that will be used in Theorem \ref{THM_ZPAIR}.
\begin{defi}\label{DEF_PAIR_COMPLETE}
	A Rauzy path on $\PAIRS{\AAA}$ is \emph{complete} if each $a\in \AAA$
		wins at least once.
	For $C\in \NN$, a path is \emph{$C$-complete}
		if it is the concatenation of $C$ Rauzy paths
		that are each complete.
\end{defi}

Recall from the introduction that an interval exchange satisfies i.d.o.c.
	if and only if Rauzy induction may be applied to it infinitely many times.
We include a formal statement here which may be verified in \cite[Section 4]{cYoccoz} for example.

\begin{prop}
	For fixed infinite Rauzy path on $\PAIRS{\AAA}$ the following are equivalent:
	\begin{enumerate}
		\item The path was obtained by induction on an i.d.o.c. interval exchange transformation, and
		\item For any fixed $j\in \NN$ there exists $j'\in \INTio{j+1,\infty}$
		so that the finite path consisting of moves $j$ to $j'$ is complete.
	\end{enumerate}
\end{prop}

We then have the following.

\begin{coro}
	If an infinite Ruazy path was obtained by induction on an i.d.o.c. interval exchange transformation,
		then for any $C\in \NN$ there exists $N\in \NN$ such that the finite subpath taken from the first $N$ steps
			is $C$-complete.
\end{coro}

\begin{rema}
	As we will show in Theorem \ref{THM_ZPAIR}, we will be able to produce by Algorithm \ref{ALGO_ZORICH} the initial pair (up to inverses) of a
	finite Rauzy path provided that it is $C$-complete for $C$ large enough.
The size of our $C$ depends on the size $|\AAA|$, and checking for $C$-completeness may be done
	using only the matrices given by induction.
	So we may take an infinite path, select $N$ so that the first $N$ steps form such a $C$-complete path.
\end{rema}

\subsection{Rauzy Induction on Permutations}\label{SSEC_RAUZY_DEF2}

In this section we discuss Rauzy induction using
	another prevalent convention for interval exchanges,
	which significantly predates
	the notation from the previous section.
We translate some of the commonly used symbols in this paper to match
	with the previous section.

\begin{defi}\label{DEF_PERM_IRR}
For $n\in \NN$, 
	a \emph{permutation} on $n$ symbols is a bijection $\pi$ from the set $\INTi{1,n}$ to itself
	and the set of all permutations on $n$ symbols is $\PERM{n}$.
A permutation $\pi\in \PERM{n}$ is \emph{irreducible}
	if
		$$
			\pi^{-1}(\INTi{1,k}) \neq \INTi{1,k} \mbox{ for all } k\in \INTi{1,n-1},
		$$
	meaning the first $k$ symbols does not remain preserved as a set
		by $\pi$ unless $k = n$.
The set of all irreducible permutations on $n$ symbols
	is $\IRR{n}$.
\end{defi}

If $\pi\in \IRR{n}$ then the result of applying Rauzy induction of type $0$
	yields $\pi'$ where
		\begin{equation}
			\pi'(i) = \begin{cases}
				\pi(i), & \mbox{if }\pi(i)\in \INTi{1,\pi(n)},\\
				\pi(i) + 1, & \mbox{if } \pi(i) \in \INTi{\pi(n)+1,n-1},\\
				\pi(n) + 1, & \mbox{if } \pi(i) = n,
				\end{cases}
		\end{equation}
and the associated matrix $A$ is given by
	\begin{equation}\label{EQ_A0_PERMS}
		A(i,j) = \begin{cases}
				1, & \mbox{if }i=j,\\
				1, & \mbox{if }i =n \mbox{ and } j = \pi^{-1}(n),\\
				0, & \mbox{otherwise}.
			\end{cases}
	\end{equation}
The result of applying Rauzy induction of type $1$
	yields $\pi''$ where
		\begin{equation}
			\pi''(i) = \begin{cases}
				\pi(i), & \mbox{if } i\in \INTi{1,\pi^{-1}(n)},\\
				\pi(n), & \mbox{if } i = \pi^{-1}(n)+1,\\
				\pi(i-1), & \mbox{if } i\in \INTi{\pi^{-1}(n)+2,n},
				\end{cases}
		\end{equation}
and the associated matrix $A$ is given by
	\begin{equation}\label{EQ_A1_PERMS}
		A(i,j) = \begin{cases}
				1, & \mbox{if } i=j\mbox{ and } i\leq \pi^{-1}(n),\\
				1, & \mbox{if } j = i+1 \mbox{ and } j > \pi^{-1}(n),\\
				1, & \mbox{if } i=n \mbox{ and } j = \pi^{-1}(n)+1,\\
				0, & \mbox{otherwise}.
				\end{cases}
	\end{equation}

Similar to the discussion before Definition \ref{DEF_PAIR_PATH}
	concerning pairs, if we know two of the following concerning a move of induction:
	the initial permutation $\pi$, the resulting permutation $\pi'$
		or the type $t$,
	then we know all three.
Furthermore, by irreducibility $\pi^{-1}(n) \neq n$ and $\pi(n) \neq n$ and so
	we may determine from a given matrix $A$ the type $t$ and the value $\pi^{-1}(n)$.
	
\begin{defi}
	A \emph{Rauzy path} or \emph{path of Rauzy induction} on $\IRR{n}$
		is an ordered (finite or infinite) sequence of moves of Rauzy induction.
	The length $N\in \NN\cup\{\infty\}$ is the 
		number of moves and the permutations visited by the path
		is a sequence $\pi^{(j)}$, $j\in \INTio{1,N+2}$,
		in $\IRR{n}$ so that $\pi^{(j+1)}$ is the
		result of applying the $j^{th}$ inductive move
			on $\pi^{(j)}$.
	We say that the path \emph{starts at } $\pi^{(1)}$ and refer to
		this as the \emph{initial permutation}.
\end{defi}

For our question, we know the matrices $A_j$, $j\in \INTi{1,N}$,
	defined by a path of Rauzy induction on $\IRR{n}$
	and want to determine the permutations visited by the path.
Over $\IRR{n}$ we do not have an analogue to Proposition \ref{PROP_PAIRS_BOTH_TYPES}
	as the type may be determined by the matrices.
Therefore, we do not have such an obstruction to determining the initial
	permutation uniquely (compare Theorem \ref{THM_ZPAIR} with Corollary \ref{COR_ZPERM}).
	
As with Definition \ref{DEF_PAIR_INVERSE}, we give the following with motivation explained in the next section by Remark \ref{REM_COMPLETE_DEF}.
\begin{defi}\label{DEF_IRR_COMPLETE}
	A Rauzy path $A_1,\dots, A_N$ on $\IRR{n}$ is \emph{complete}
		if the product matrix $\prod_{j=1}^N A_j$ has in each row at least two nonzero entries.
	For $C\in \NN$, a path is $C$-complete if it is the concatenation of $C$
		Rauzy paths that are each complete.
\end{defi}

\subsection{Translating between Pairs and Permutations}\label{SSEC_PERM+PAIRS}

In this section, we make precise the relation between matrices for Rauzy induction on $\PAIRS{\AAA}$
	and those on $\IRR{n}$.
The primary result is the statement in Lemma \ref{LEM_LIFTING_UP}, which tells us that
	a Rauzy path on $\IRR{n}$ may be lifted to a Rauzy path on $\PAIRS{\AAA}$
		even if we do not know the permutations associated to the path.
The specific contents of this section may be skipped on first reading if preferred
	as the specific notation is not as important to the main results that follow.

For finite alphabet $\AAA$ with $n = |\AAA|$,
	there is a natural map $\Pi : \PAIRS{\AAA} \to \IRR{n}$ given by
	\begin{equation}
		\Pi(p_0,p_1) := p_1 \circ p_0^{-1}.
	\end{equation}

\begin{rema}\label{REM_INVERSE_DEF}
	If $\pi = \Pi(p_0,p_1)$, then $\Pi(p_1,p_0)$ is equal to $\pi^{-1}$.
	This justifies the definition of inverese for pairs in Definition \ref{DEF_PAIR_INVERSE}.
\end{rema}

Given $(q_0,q_1)\in \PAIRS{\BBB}$ over alphabet $\BBB$, $|\BBB| = n$,
	we have that $\Pi(q_0,q_1) = \Pi(p_0,p_1)$
	if and only if there exists a bijection $\tau :\BBB \to \AAA$
	such that
	$q_t = p_t \circ \tau$ for each $t\in \{0,1\}$.
	We use the same bijection name to define a map from $\PAIRS{\AAA}$ to $\PAIRS{\BBB}$
		by
	$$
		\tau(p_0,p_1) := (p_0\circ \tau, p_1 \circ \tau).
	$$
There is also a natural lift $\widehat{\Pi}: \IRR{n} \to \PAIRS{\INTi{1,n}}$
	given by
	\begin{equation}
		\widehat{\Pi}(\pi) := (\ID{n},\pi),
	\end{equation}
	where $\ID{n}\in \PERM{n}$ is the identity.
	
Fix $(p_0,p_1)\in \PAIRS{\AAA}$.
If we let $\RRR_t$ denote the action of Rauzy induction of type $t$ (on pairs and permutations),
	we want to explicitly describe the relationship between
	$(p_0',p_1'):= \RRR_t(p_0,p_1)$ and $\pi' := \RRR_t(\pi)$ for
		$\pi := \Pi(p_0,p_1) \in \IRR{n}$.
Let
$$
	(q_0,q_1) := \RRR_t \big(\widehat{\Pi}(\pi)\big) = \RRR_t(\ID{n},\pi).
$$
For these choices we then define bijections $\tau,\tau':\INTi{1,n} \to \AAA$ and $\sigma_t:\INTi{1,n} \to \INTi{1,n}$ by
$$
		\tau(p_0,p_1) = \widehat{\Pi}(\pi),~\tau'(p_0',p_1') = \widehat{\Pi}(\pi')\mbox{ and }
		\sigma_t(q_0,q_1) = \widehat{\Pi}(\pi').
$$
We may directly verify that $\tau = p_0^{-1}$ and $\tau' = p_0'^{-1}$,
	while $\sigma_t$ will depend on the type $t$ and possibly the value $k = \pi^{-1}(n)$.
The relationships between these pairs and bijections are provided by diagram in Figure \ref{FIG_COMMUTE}.
We are equivalently defining $\sigma_t$ as the bijection (acting on pairs) such that
	$$
		\widehat{\Pi} \circ \RRR_t = \sigma_t \circ \RRR_t\circ \widehat{\Pi}.
	$$
Because induction commutes with any renaming bijection on pairs, this then implies
as bijections on $\AAA$ we have
	$
		\tau' = \tau \circ \sigma_t
	$
,
meaning
	$$
		\tau'(j) = \tau(\sigma_t(j)) \mbox{ for all } j\in \INTi{1,n}.
	$$
(We note that we have by definition $\sigma_t\big(\tau (p_0,p_1)\big)$
	equals $\sigma_t(p_0\circ\tau,p_1\circ \tau)$
	and is therefore $(p_0\circ \tau \circ \sigma_t, p_1 \circ \tau \circ \sigma_t) = (\tau \circ \sigma_t) \big(p_0,p_1\big)$.)

If $t = 0$, then $q_0 = \ID{n}$ and $q_1 = \pi'$ and so $\sigma_0 = \ID{n}$.
If $t = 1$, then $q_1 = \pi$ but $q_0 \neq \ID{n}$.
We may verify that $\sigma_1$ is given by
	$$
		\sigma_1(j) = \begin{cases}
			j, & j\in \INTi{1,k},\\
			n, & j = k+1,\\
			j-1, & j \in \INTi{k+2,n},
			\end{cases}
	$$
where $k = \pi^{-1}(n)$.

If $\Theta$ is the matrix from \eqref{EQ_THETA_PAIRS} applied to $(p_0,p_1)$
 and $A$ is the matrix from \eqref{EQ_A0_PERMS} or \eqref{EQ_A1_PERMS}
 applied to $\pi$, then we have
 \begin{equation}
 	\Theta = \Psi_{\tau} A \Psi^*_{\tau\circ \sigma_t}
 \end{equation}
 where $\Psi_\tau$ is the $\INTi{1,n}\times\AAA$ matrix defined by
\begin{equation}\label{EQ_Psi_tau_def}
	\Psi_\tau(j,b) = \begin{cases}
			j, & \tau(i) = b,\\
			0, & \tau(b) \neq j,
		\end{cases}
\end{equation}
and likewise for $\Psi_{\tau \circ \sigma_t}$.

\begin{rema}
Before continuing, we make a few observations.
First, the adjoint $\Psi^*_{\tau\circ \sigma_t}$
	is the inverse of $\Psi_{\tau\circ \sigma_t}$, meaning
	$\Psi_{\tau\circ \sigma_t}\Psi^*_{\tau\circ \sigma_t}$ is the $\INTi{1,n}\times \INTi{1,n}$
	identity matrix and $\Psi^*_{\tau\circ \sigma_t}\Psi_{\tau\circ \sigma_t}$ is the $\AAA\times \AAA$
		identity matrix.
Second, if we act on $\pi'$ and $(p_0',p_1')$ by induction $\RRR_{t'}$,
	the new $\tau'$ we obtain would precisely be $\tau \circ \sigma_t$.
Finally, if we know the matrix $A$, we know the type $t$ and $k = \pi^{-1}(n)$
	if $t=1$ therefore knowing $A$ and $\tau$ uniquely determines $\Theta$
	(even without knowing $\pi$ or $\pi'$).
\end{rema}

\begin{figure}[t]
	\begin{center}
		\includegraphics[scale=1]{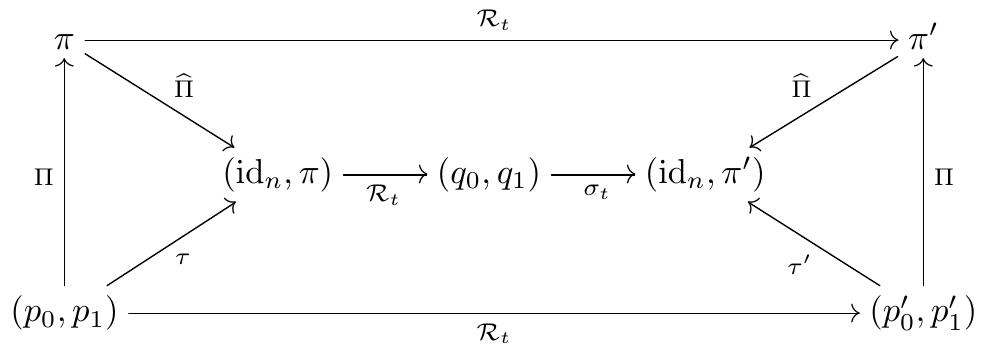}
	\end{center}
\caption{Diagram illustrating the relationship between the pairs $(p_0,p_1)$, $(q_0,q_1)$ and the permutation $\pi$.}\label{FIG_COMMUTE}
\end{figure}

The next result will be used to take our main result on $\PAIRS{\AAA}$, Theorem \ref{THM_ZPAIR},
	to arrive at our main result on $\IRR{n}$, Corollary \ref{COR_ZPERM}.
The proof follows by our work in this section and induction.

\begin{lemm}\label{LEM_LIFTING_UP}
	Let matrices $A_1,\dots, A_N$ be associated to a Rauzy path in $\IRR{n}$ and
		$\tau$ be any bijection from $\INTi{1,n}$ to an alphabet $\AAA$.
	If $t_i$ is the type of $A_i$ and $k_i$ satisfies $\pi^{(i)}(n) = k_i$ for $i$ such that $t_i = 1$,
		let $\Theta_1,\dots,\Theta_N$ be defined by
			$$
				\Theta_i = \Psi_{\tau_i} A_i \Psi^*_{\tau_{i+1}} \mbox{ for } i\in \INTi{1,N}
			$$
		where $\tau_1 = \tau$ and $\tau_{i+1} = \tau_i \circ \sigma_{t_i}$ for $i\in \INTi{1,N}$.
	Then if $(p_0,p_1)\in \PAIRS{\AAA}$ begins a Rauzy path described by $\Theta_1,\dots,\Theta_N$ then
		exactly one of the following must hold:
	\begin{enumerate}
		\item For each $i\in \INTi{1,N}$ the move described by $\Theta_i$ is type $t_i$ and $\Pi(p_0,p_1)\in \IRR{n}$
			starts a Rauzy path described by $A_1,\dots,A_N$.
		\item For each $i\in \INTi{1,N}$ the move described by $\Theta_i$ is type $1-t_i$ and $\Pi(p_1,p_0)\in \IRR{n}$
			starts a Rauzy path described by $A_1,\dots,A_N$.
	\end{enumerate}
	Conversely, if $\pi\in \IRR{n}$ starts a Rauzy path described by $A_1,\dots,A_N$ then
		both
			$$(\tau,\pi\circ \tau)\mbox{ and }(\pi\circ \tau,\tau)$$
		start Rauzy paths described by $\Theta_1,\dots,\Theta_N$,
		one with types that agree with the path starting at $\pi$ and the other with opposite types.
\end{lemm}

\begin{rema}\label{REM_COMPLETE_DEF}
	The descriptions of complete and $C$-complete from Definitions \ref{DEF_PAIR_COMPLETE} and \ref{DEF_IRR_COMPLETE}
		are equivalent in the following sense:
		a path $A_1,\dots,A_N$ over $\IRR{n}$ is $C$-complete if and only if any lift to $\Theta_1,\dots,\Theta_N$
			over an alphabet $\AAA$ is $C$-complete.
	This is because the transformations $\Psi_{\tau_i}$ and $\Psi^*_{\tau_i}$ only
		change the indexing of rows and columns but keeps entries in the same row together (and likewise for columns).
\end{rema}

	\subsection{Zorich Induction on Pairs}\label{SSEC_ZORICH_DEF}
			
			Zorich induction on $\PAIRS{\AAA}$ is the acceleration of Rauzy induction that applies all consecutive steps of Rauzy induction in a path of the same
				type as one action, and therefore one matrix which we denote by $\widetilde{\Theta}$.
			We obtain this matrix by taking the product of the successive $\Theta$ matrices from Rauzy induction.
			As we have discussed, the actual type of a matrix $\Theta$ cannot in itself be determined.
			However, two consecutive matrices in a Rauzy path have the same type if and only if the ``winner row'',
				meaning the only row with a non-zero off-diagonal entry,
				of each matrix are the same index $w\in \AAA$.
			The losers will likely be different.
			
			\begin{exam}
				If we have initial pair $(p_0,p_1)$ given by
					\begin{equation}
						(p_0,p_1) = \left(\begin{matrix}
							1&2&3&4&5\\5&4&3&2&1
							\end{matrix}\right)
					\end{equation}
				and the Rauzy path starts with six steps of type $1$ (and the seventh is type $0$), then
				the resulting matrix for this one step of Zorich induction is
					\begin{equation}\label{EQ_ZORICH_EXAM}
						\widetilde{\Theta} = \left(\begin{matrix}
												1 & 1 & 1 & 2 & 2\\
												0 & 1 & 0 & 0 & 0\\
												0 & 0 & 1 & 0 & 0\\
												0 & 0 & 0 & 1 & 0\\
												0 & 0 & 0 & 0 & 1
												\end{matrix}\right).
					\end{equation}
				Note that for each letter $j\in \INTi{2,5}$, the entry $\widetilde{\Theta}(1,j)$
					gives us the number of steps in which $j$ lost to $1$.
				Furthermore, each such letter $j$ lost to $1$ once before any letter lost to $1$ a second time.
			\end{exam}
			
			In this paper, we will somewhat relax this acceleration to allow for any products of consecutive moves of the same type
				as a Zorich move.
			Specifically, we will allow for two consecutive Zorich matrices of the same
				type to exist in a path rather than forcing the type to change between each step.
				
			\begin{defi}
				A sequence of $\AAA\times \AAA$ matrices $\big(\widetilde{\Theta}_n\big)_{n\in \INTi{1,N}}$ of length $N\in \NN\cup\{\infty\}$ is a
					\emph{Zorich path} on $\PAIRS{\AAA}$ if there exists an \emph{underlying Rauzy path} $\big(\Theta_m\big)_{m\in \INTi{1,N'}}$ of length $N' \geq N$
						($N'<\infty$ if $N<\infty$)
					such that for an increasing sequence of numbers $1=m_1<m_2<\dots < N'+1$ (with $m_{N+1} = N' +1$ if $N  \leq N' <\infty$) we have that
					for each $n\in \INTio{1,N+1}$
					all matrices $\Theta_i$ for $i\in \INTio{m_{n},m_{n+1}}$ have the same type and
					$\widetilde{\Theta}_n = \Theta_{m_n} \Theta_{m_{n}+1} \cdots \Theta_{m_{n+1}-1}$.
			\end{defi}
			
			\begin{exam}
				Continuing along our previous example,
					we could also express the six Rauzy moves
					as two consecutive Zorich matrices,
					\begin{equation}\label{EQ_ZORICH_EXAM2}
						\widetilde{\Theta}_1 = \left(\begin{matrix}
												1 & 1 & 1 & 1 & 1\\
												0 & 1 & 0 & 0 & 0\\
												0 & 0 & 1 & 0 & 0\\
												0 & 0 & 0 & 1 & 0\\
												0 & 0 & 0 & 0 & 1
												\end{matrix}\right),
											~
						\widetilde{\Theta}_2 = \left(\begin{matrix}
												1 & 0 & 0 & 1 & 1\\
												0 & 1 & 0 & 0 & 0\\
												0 & 0 & 1 & 0 & 0\\
												0 & 0 & 0 & 1 & 0\\
												0 & 0 & 0 & 0 & 1
												\end{matrix}\right),
					\end{equation}
				by taking $\widetilde{\Theta}_1$ to be the product of the first four Rauzy matrices and
					$\widetilde{\Theta}_2$ to be the product of the last two.
				(There are certainly many more choices.)
			\end{exam}

			It follows that knowing only a Zorich path can result in a loss of information.
			For example, if we only know $\widetilde{\Theta}_1$ from \eqref{EQ_ZORICH_EXAM2}, we
				know that: four moves of Rauzy induction occurred, the winner is $1$ and the losers were
				$2,3,4,5$ each one time.
			However, just from this matrix we are not able to determine the order in which the losses occurred.
			We know a little more from $\widetilde{\Theta}$ in \eqref{EQ_ZORICH_EXAM}.
			In this case, $2,3,4,5$ were the losers, but after each letter lost for the first time
				 the letters $4,5$ then lost a second time.
			We therefore know that $2$ and $3$ must have occurred before $4$ and $5$ in the row $1-t$ of the pair before
				this move,
				where $t$ is the assumed type.
			
			\begin{defi}
				A Zorich path is \emph{$C$-complete} if there exists an underlying $C$-complete Rauzy path.
			\end{defi}
			
			\begin{rema}
			While it is possible that more than one underlying path exists for a Zorich path,
				it is left as an exercise to verify $C$-completeness is independent of the Rauzy path selected.
			Furthermore, we may verify the $C$-completeness of a Zorich path without knowing an underlying Rauzy path.
			We can do this because we may recover the \emph{winners} in order and with multiplicity of any underlying Ruazy path from the
				$\widetilde{\Theta}$ matrices.
			\end{rema}
			
	\subsection{Zorich Induction on Permutations}\label{SSEC_ZORICH_DEF2}
	
		We now address our modification of Zorich acceleration on $\IRR{n}$.
		The definitions will be analogous to those in the previous section.
		We then conclude that Lemma \ref{LEM_LIFTING_UP} still holds for Zorich paths.
		
		\begin{defi}
				A sequence of $n\times n$ matrices $\big(\widetilde{A}_b\big)_{b\in \INTi{1,N}}$ of length $N\in \NN\cup\{\infty\}$ is a
					\emph{Zorich path} on $\IRR{n}$ if there exists an \emph{underlying Rauzy path} $\big(A_m\big)_{m\in \INTi{1,N'}}$ of length $N' \geq N$
						($N'<\infty$ if $N<\infty$)
					such that for an increasing sequence of numbers $1=m_1<m_2<\dots < N'+1$ (with $m_{N+1} = N' +1$ if $N  \leq N' <\infty$) we have that
					for each $b\in \INTio{1,N+1}$
					all matrices $A_i$ for $i\in \INTio{m_{b},m_{b+1}}$ have the same type and
					$\widetilde{A}_b = A_{m_b} A_{m_{b}+1} \cdots A_{m_{b+1}-1}$.
				A Zorich path is \emph{$C$-complete} if there exists an underlying $C$-complete Rauzy path.
		\end{defi}
		
		We now briefly note the structure of a matrix $\widetilde{A}$ from a Zorich move.
		If the underlying Rauzy matrices were from \eqref{EQ_A0_PERMS}, then
			$\widetilde{A}$ has a similar structure as matrix $\widetilde{\Theta}$ on $\PAIRS{\AAA}$,
			all non-zero off diagonal entries are in row $n$.
		If the underlying Rauzy matrices were from \eqref{EQ_A1_PERMS}, then
			these underlying matrices are all identical as $k=\pi^{-1}(n)$ is
				not changed by a move of type $1$.
		In this case $\widetilde{A}$ is a power of one underlying matrix.
		In either case, we know the type of $\widetilde{A}$ and if it is type $1$
			we also know $\pi^{-1}(n)$.
		
		If for fixed $\AAA$ such that $|\AAA| = n$ and bijection $\tau: \INTi{1,n} \to \AAA$,
			we recall the lifting matrix
			$\Psi_\tau$ from \eqref{EQ_Psi_tau_def} that allowed to relate a Rauzy path (denoted by $A_i$'s) on $\IRR{n}$
			to one on $\PAIRS{\AAA}$ (denoted by $\Theta_i$'s).
		In the discussion prior to Lemma \ref{LEM_LIFTING_UP}, we also defined bijections $\sigma_t$
			so that for each Rauzy matrix $A_i$ of type $t_i$
			we had that $\Theta_i = \Psi_{\tau_i} A_i \Psi^*_{\tau_{i+1}}$
			for each $i$, where $\tau_{i+1} = \tau_i \circ \sigma_{\tau_i}$ for $i\geq 1$ and $\tau = \tau_1$.
		Now assume that we have Zorich move $\widetilde{A} = A_{i_1} \cdots A_{i_2}$.
			Because the moves from $i$ to $i+1$ for $i\in \INTio{i_1,i_2}$ are all the same type, it follows that
				$$
				\Psi_{\tau_i} \widetilde{A} \Psi^*_{\tau_i \circ \sigma_t^{i_2-i_1+1}} = \widetilde{\Theta}
				$$
		where $\widetilde{\Theta} = \Theta_{i_1} \cdots \Theta_{i_2}$ is the corresponding Zorich move $\PAIRS{\AAA}$.
		We are then able to expand Lemma \ref{LEM_LIFTING_UP}, our lifting argument from the Rauzy case, here.
		
		\begin{coro}\label{COR_LIFTING_ZOR}
				Let matrices $\widetilde{A}_1,\dots, \widetilde{A}_N$ be associated to a Zorich path in $\IRR{n}$ and
		$\tau$ be any bijection from $\INTi{1,n}$ to an alphabet $\AAA$.
	If $t_i$ is the type of $\widetilde{A}_i$, $M_i$ is the number of underlying Rauzy matrices in the product $\widetilde{A}_i$,
		and $k_i$ satisfies $\pi^{(i)}(n) = k_i$ for $i$ such that $t_i = 1$,
		let $\widetilde{\Theta}_1,\dots,\widetilde{\Theta}_N$ be defined by
			$$
				\widetilde{\Theta}_i = \Psi_{\tau_i} \widetilde{A}_i \Psi^*_{\tau_{i+1}} \mbox{ for } i\in \INTi{1,N}
			$$
		where $\tau_1 = \tau$ and $\tau_{i+1} = \tau_i \circ \sigma_{t_i}^{M_i}$ for $i\in \INTi{1,N}$.
	Then if $(p_0,p_1)\in \PAIRS{\AAA}$ begins a Zorich path described by $\widetilde{\Theta}_1,\dots,\widetilde{\Theta}_N$ then
		exactly one of the following must hold:
	\begin{enumerate}
		\item For each $i\in \INTi{1,N}$ the move described by $\widetilde{\Theta}_i$ is type $t_i$ and $\Pi(p_0,p_1)\in \IRR{n}$
			starts a Rauzy path described by $\widetilde{A}_1,\dots,\widetilde{A}_N$.
		\item For each $i\in \INTi{1,N}$ the move described by $\widetilde{\Theta}_i$ is type $1-t_i$ and $\Pi(p_1,p_0)\in \IRR{n}$
			starts a Rauzy path described by $\widetilde{A}_1,\dots,\widetilde{A}_N$.
	\end{enumerate}
	Conversely, if $\pi\in \IRR{n}$ starts a Rauzy path described by $\widetilde{A}_1,\dots,\widetilde{A}_N$ then
		both
			$$(\tau,\pi\circ \tau)\mbox{ and }(\pi\circ \tau,\tau)$$
		start Rauzy paths described by $\widetilde{\Theta}_1,\dots,\widetilde{\Theta}_N$,
		one with types that agree with the path starting at $\pi$ and the other with opposite types.

		\end{coro}

	\subsection{Breaking Up Zorich Matrices}\label{SSEC_BREAKUP}
	
		If $\widetilde{\Theta}$ is a matrix describing a move of Zorich induction in $\AAA$,
			the winner $a$ will be uniquely identified by the row with
			at least one positive entry other than the diagonal $\widetilde{\Theta}(a,a)$.
		The loser set $\LLL$ is then
			$$
				\LLL := \{b\in \AAA\setminus\{a\}: ~\widetilde{\Theta}(a,b) \geq 1\}.
			$$
		Let $M = \max\{\widetilde{\Theta}(a,b):~a,b\in\AAA\}$.
		If $M \geq 2$, let
			$$
				\LLL_{max} := \{b\in \LLL:~ \widetilde{\Theta}(a,b) = M\}.
			$$
			Because we know that a step of Zorich induction
			must apply Rauzy induction to each element of $\LLL$ in some order
			before applying induction to any element again, $\LLL_{min}:= \LLL \setminus \LLL_{max}$
			satisfies
			$$
			 \LLL_{min} = \{b\in \LLL:~\widetilde{\Theta}(a,b) = M-1\}.
			$$
		
		If $\LLL_{min}\neq \emptyset$, we know that the move $\widetilde{\Theta}$
			was composed of $M-1$ cycles of Rauzy induction, each move with winner $a$
				and having each element of $\LLL$ as a loser in an unknown order once per cycle.
		The Zorich move then ended with a parital cycle of Rauzy moves,
			$a$ being the winner and each letter of $\LLL_{max}$ losing in the same (unknown) order as in the previous cycle.
		(Therefore, in the cycles the elements of $\LLL_{max}$ lost before the elements of $\LLL_{min}$.)
		In this case, let $\Lambda$ and $\Lambda_{min}$ be $\AAA\times \AAA$ matrices given by
			$$
				\Lambda(\alpha,\beta) = \begin{cases}
							1, & \mbox{if }\alpha = \beta,\\
							1, & \mbox{if }\alpha = a\mbox{ and }\beta\in \LLL,\\
							0, & \mbox{otherwise.}
							\end{cases}~
				\Lambda_{max}(\alpha,\beta) = \begin{cases}
							1, & \mbox{if }\alpha = \beta,\\
							1, & \mbox{if }\alpha = a\mbox{ and }\beta\in \LLL_{max},\\
							0, & \mbox{otherwise.}
							\end{cases}
			$$
		Then
			$
			\widetilde{\Theta} = \Lambda^{M-1} \Lambda_{max}
			$
		and we may use these $M$ matrices in place of $\widetilde{\Theta}$
			in any Zorich path.
		
		By applying this splitting as needed, we may ensure that our Zorich path
			has all matrices with entries at most $1$.			
		Furthermore, if for some Zorich move on $\IRR{n}$
			described by $\widetilde{A}$ we have type $t$ and bijection $\tau$
				so that 
			$$
				\widetilde{\Theta} = \Psi_\tau \widetilde{A} \Psi^*_{\tau'}
			$$
			where $\tau' = \tau \circ \sigma_t^p$ where $p$ is the number of steps of Rauzy induction that
			$\widetilde{A}$ represents.

	\section{Main Result}\label{SEC_MAIN1}
	
	In this section we will present our first main result.
	After defining partially ordered pairs in Section \ref{SSEC_POPAIRS},
		we will justify and construct Algorithm \ref{ALGO_ZORICH}
		in Section \ref{SSEC_ALGO_DEF}
		with example runs of the algorithm in the section that follows.
	In Section \ref{SSEC_MAIN1_PROOF} we then prove the main theorem.
	
	\subsection{Partially Ordered Pairs}\label{SSEC_POPAIRS}
	
	In our upcoming algorithm, we need to track the information we
		currently have about our pairs.
	The following will help make this precise.
	
		\begin{defi}\label{DEF_POPAIRS}
		For finite $\AAA$,
			a \emph{partially ordered pair}
			is a tuple $(\QQQ_0,\QQQ_1)$
			such that, for each $t\in \{0,1\}$,
			$$\QQQ_t = \big(Q_{t,1},\dots, Q_{t,{m_t}}\big),$$
			is a tuple of subsets of $\AAA$, where $m_t:=|\QQQ_t|$, such that
			$\{Q_{t,j}:~j\in \INTi{1,m_t}\}$ is a partition of $\AAA$.
	\end{defi}
	
	We want $(\QQQ_0,\QQQ_1)$ to be our best current guess at
		the underlying pair $(p_0,p_1)\in \PAIRS{\AAA}$,
		and the following definition makes this more precise.
		
	\begin{defi}
		A pair $(p_0,p_1)$ over $\AAA$, with $n = |\AAA|$,
			\emph{agrees with partially ordered pair} $(\QQQ_0,\QQQ_1)$
			if for each $t\in \{0,1\}$ and $Q \in \QQQ_t$ the set $p_t(Q) \subseteq \INTi{1,n}$
			is an interval and for $\alpha,\beta\in \AAA$ if $\alpha\in Q_{t,j}$
			and $\beta\in Q_{t,j'}$ such that $j< j'$ we have $p_t(\alpha) < p_t(\beta)$.
		A partially ordered pair $(\QQQ_0,\QQQ_1)$ is \emph{irreducible}
			if there exists $(p_0,p_1)\in \PAIRS{\AAA}$ that agrees with $(\QQQ_0,\QQQ_1)$.
	\end{defi}

	As we build our algorithm in the next section,
		we will often need to refine our ordered partitions.
	To avoid too many subcases,
		we may end up defining a new tuple $\QQQ_t'$ that possibly contains
		entries that are the empty set $\emptyset$.
	We will always assume that we remove these entries in $\QQQ_t'$
		but to reinforce this with notation we will define the map
		$\star$ that removes these elements.
	For a fixed tuple $(Q_1,\dots, Q_m)$, $Q_i\subseteq \AAA$ for each $i\in \INTi{1,m}$,
		define $\mu$ with domain $\INTi{0,m}$ by
		$$
			\mu(i) := \begin{cases}
				|\{j \in \INTi{1,i}:~Q_j \neq \emptyset\}|, & i \in \INTi{1,m},\\
				0, & i = 0,
				\end{cases}
		$$
		and let $m' = \mu(m)$.
	Then define $\xi:\INTi{1,m'} \to \INTi{1,m}$ by
		$$
			\xi(i) = j \iff \mu(j) = i \mbox{ and } \mu(j-1) = i-1.
		$$
	We then define our map $\star$ as
			\begin{equation}\label{EQ_STAR}
				\star\big(Q_1,\dots,Q_m\big) = \big(Q_{\xi(1)},\dots,Q_{\xi(m)}\big),
			\end{equation}
	assuming that at least one $Q_i$ is non-empty.

	\subsection{The Algorithm}\label{SSEC_ALGO_DEF}

	We now construct and justify our algorithm.
	Assume $n = |\AAA|$ is the number of letters in our alphabet with $n\geq 3$.
	We consider the information we learn about our candidate pair $(p_0,p_1)$
		given knowledge of \emph{the resulting} pair $(p_0',p_1')$,
		the winner $w$ and the loser set $\LLL$.
	We also assume for the moment that we know the type $t$ of the move (we will address this below).
	Then $w$ must the rightmost element in both $p_t$ and $p_t'$ as $p_t = p_t'$.
	The letters $\LLL$ must be the rightmost $|\LLL|$
		elements of $p_{1-t}$.
	They also must be the letters to the immediate right of $w$ in $p_{1-t}'$,
		and the letters in $\LLL$ must appear consecutively and in the same relative order
		in both $p_{1-t}$ and $p_{1-t}'$.
	Note that, aside from the previous observations,
		we may not know very much about $p_{1-t}'$.
	In particular, while we know that $w$ appears with the letters of $\LLL$
		to the right, the position of $w$ may be unknown.
	However, we do not need this information to place $\LLL$ as the rightmost letters
		in $p_{1-t}$.
		
	Given that we gain information about the pre-image of
		a move of Zorich induction, we will iterate through our Zorich path
		described by
		$$
			(w_1,\LLL_1),\dots,(w_N,\LLL_N),
		$$
	\emph{in reverse},
		meaning we use the information from the move, $w_j,\LLL_j$,
		and what we currently know about $(p_0^{(j+1)},p_1^{(j+1)})$,
		given by $(\QQQ_0^{(j+1)},\QQQ_1^{(j+1)})$,
		to construct $(\QQQ_0^{(j)},\QQQ_1^{(j)})$
			which represents what we currently know about $(p_0^{(j)},p_1^{(j)})$.
	We will know the initial pair $(p_0,p_1)$, up to switching rows,
		if $(\QQQ_0,\QQQ_1)$ is such that $|\QQQ_t| = n$ for each $t\in \{0,1\}$
		or equivalently $\QQQ_t$ is an ordered partition of $\AAA$ into singletons.
	
	For the last move $(w_N,\LLL_N)$ we do not know anything about the
		final pair
		$(p_0^{(N+1)},p_1^{(N+1)})$,
		or
			$$\QQQ_0^{(N+1)} = \QQQ_1^{(N+1)} = \{\AAA\}.$$
	While we cannot know the type $t_N$ of this move,
		by Proposition \ref{PROP_PAIRS_BOTH_TYPES}
		there would exist a Zorich path with one type $t_N$
		and another with the opposite type $1-t_N$.
	So we choose, say, $t_N=0$
		with the understanding that the choice $t_N =1$
		would result another path on the inverse initial pair.
	Because $w_N$ is the winner, we then may say with certainty
		that $w_N$ is the rightmost letter in $\QQQ^{(N)}_0$
		and $\LLL_N$ as a set must be the rightmost letters in $\QQQ_1^{(N)}$,
		although we do not know their relative order unless $|\LLL_N|=1$.
	(This is Step \stepref{ALG2_STEP_1} in Algorithm \ref{ALGO_ZORICH}.)

	Now for $j\in \INTi{1,N-1}$, we will use our knowledge of the type $t_{j+1}$
		of the most recently considered move and the partially ordered pair
		$(\QQQ_0^{(j+1)},\QQQ_1^{(j+1)})$ to determine the type $t_j$ and construct
		$(\QQQ_0^{(j)},\QQQ_1^{(j)})$. 
	First, if we have that the winner $w_j$ is the same as the previously considered winner $w_{j+1}$,
		then by irreducibility the type $t_j$ must be the same as $t_{j+1}$.
	Conversely, if $w_j \neq w_{j+1}$ then it must be that $t_j$ is $1 - t_{j+1}$,
		the opposite type.
	(This is Step \stepref{ALG2_STEP_2a} in Algorithm \ref{ALGO_ZORICH}.)
	In either case, it must be that $w_j$ is in the rightmost element of $\QQQ_{t_j}^{(j+1)}$
		provided the assumed correctness of the given Zorich path.
	We then construct $\QQQ^{(j)}_{t_j}$ by preserving the order of $\QQQ_{t_j}^{(j+1)}$ except
		we possibly split the rightmost element into two: $\{w_j\}$ being the rightmost element
		with its complement (if non-empty) being to its immediate left.
	(This is Step \stepref{ALG2_STEP_2b} in Algorithm \ref{ALGO_ZORICH}.)

	We need to consider cases in order to construct $\QQQ_{1-t_j}^{(j)}$.
	Suppose first that the set $\LLL_j$
		is not contained in one element of $\QQQ_{1-t_j}^{(j+1)}$.
	There then must be two or more consecutive elements in $\QQQ_{1-t_j}^{(j+1)}$
		that have non-empty intersection with $\LLL_j$
		and so the position of the set $\LLL_j$ in the \emph{previously considered} $p_{1-t_j}^{(j+1)}$
		is now exactly known. (Meaning we know the first and last positions of $\LLL_j$ as a set
			but may not know the exact order of each $\ell\in \LLL_j$ in $p_{1-t_j}^{(j+1)}$.)
	Let $Q_L$ be the leftmost element in $\QQQ_{1-t_j}^{(j+1)}$ that overlaps
		with $\LLL_j$ and left $Q_R$ be the rightmost element.
	We know that in $\QQQ_{1-t_j}^{(j+1)}$, the elements of $Q_R$
		actually appear in the order $Q_R\cap \LLL_j$ and then $Q_R \setminus \LLL_j$.
	(It is possible that $Q_R\setminus \LLL_j = \emptyset$ but $Q_R\cap \LLL_j$ is non-empty by definition.)
	As we will now explain, we also know the position of $w_j$ in $\QQQ_{1-t_j}^{(j+1)}$.
	If $w_j \in Q_L$, then we know that the elements of $Q_L$ in $\QQQ_{1-t_j}^{(j+1)}$
		actually appear in the order $Q_L\setminus(\LLL_j \cup \{w_j\})$, $\{w_j\}$, $Q_L \cap \LLL_j$
			(from left to right).
	(It is possible that $Q_L\setminus(\LLL_j \cup \{w_j\}) = \emptyset$ but the others must be non-empty.)
	If $w_j\not\in Q_L$, then because $w_j$ must be the letter to the immediate left
		of $\LLL_j$ in $\QQQ_{1-t_j}^{(j+1)}$, it must be that $Q_L \subsetneq \LLL_j$
		and $w_j$ is rightmost element of $Q_{L}'$, the element of $\QQQ_{1-t_j}^{(j+1)}$ to the
		immediate left of $Q_L$.
		
	Because of these new divisions, every element of $\QQQ_{1-t_j}^{(j+1)}$ is now either contained
		in $\LLL_j$ or disjoint from $\LLL_j$.
	We create $\QQQ_{1-t_j}^{(j)}$
		by leaving moving all elements contained in $\LLL_j$ to be the rightmost
		elements while preserving their relative order.
	(This process is given by Steps \stepref{ALG2_STEP_2f} and \stepref{ALG2_STEP_2g}, depending
		on whether $w_j \in Q_L$ or $w_j \notin Q_L$.)

	In the remaining case $\LLL_j$ is contained in one element $Q$ of $\QQQ_{1-t_j}^{(j+1)}$.
	Just as before, either $w_j\in Q$ or $w_j\not\in Q$.
	If $w_j\not\in Q$, then we know that $\LLL_j$ forms the leftmost letters of $Q$
		and $w_j$ is the rightmost letter of $Q'$, the element in $\QQQ_{1-t_j}^{(j+1)}$
			to the immediate left of $Q$.
	We then may divide $Q'$ in $\QQQ_{1-t_j}^{(j+1)}$
		into $Q'\setminus \{w_j\}$ on the left and $\{w_j\}$ on the right,
			and we may divide $Q$ into $\LLL_j$ on the left and $Q\setminus \LLL_j$
			on the right.
	(It may be that $Q'\setminus \{w_j\}$, $Q\setminus \LLL_j$ or both are empty.)
	We then construct $\QQQ_{1-t_j}^{(j)}$ by moving $\LLL_j$ to be the rightmost element.
	(This is Step \stepref{ALG2_STEP_2e}.)

	If instead $w_j\in Q$, then we cannot necessarily determine the location of $w_j$
		or $\LLL_j$ within $Q$.
	We then construct $\QQQ_{1-t_j}^{(j)}$
		by making $\LLL_j$ the rightmost element and leaving (non-empty) $Q\setminus \LLL_j$
		in its relative position according to $\QQQ_{1-t_j}^{(j+1)}$.
	(This is Step \stepref{ALG2_STEP_2d}.)
	
	As a note to the cautious reader regarding the previous step,
		it is possible that the true division of $Q$ in $\QQQ_{1-t_j}^{(j+1)}$
		is, from left to right, $Q'$, $\{w_j\}$, $\LLL_j$ then $Q''$ with $Q',Q''$ non-empty.
	However, when we move $\LLL_j$ to the right, the remaining elements still appear in the order
		$Q'$, $\{w_j\}$ then $Q''$.
	So even though we may not know $Q'$ and $Q''$, we know the location of $Q \setminus \LLL_j = Q' \cup \{w_j\}\cup Q''$ in
		newly constructed $\QQQ_{1-t_j}^{(j)}$.
		
	We now present Algorithm \ref{ALGO_ZORICH} as just explained.
	To allow for independent reading of the algorithm,
		notation differs slightly from the previous discussions.
	Recall the $\star$ operation removes empty set valued coordinates from a tuple.

	\begin{algo}\label{ALGO_ZORICH}
	Assume $|\AAA| \geq 3$ and consider the list of winner letter/loser set tuples
		$$
			(w_1,\LLL_1),\dots, (w_N,\LLL_N)
		$$
	of a Zorich path of length $N\geq 1$.
	\begin{enumerate}
		\renewcommand{\theenumi}{\texttt{\arabic{enumi}}}
		\renewcommand{\labelenumi}{\texttt{\color{red}[\theenumi]}}
		
		\item\label{ALG2_STEP_1} Let $t_N= 0$,
			$$
				\QQQ_0^{(N)} = \big(\AAA\setminus \{w_N\}, \{w_N\}\big) \mbox{ and }
				\QQQ_1^{(N)} = \big(\AAA\setminus \LLL_N, \LLL_N\big).
			$$
			
		\item\label{ALG2_STEP_2} Starting with $j = N-1$ and iterating down to $j=1$:
			\begin{enumerate}
				\renewcommand{\theenumii}{\texttt{.\arabic{enumii}}}
				\renewcommand{\labelenumii}{\texttt{\color{red}[\theenumi\theenumii]}}

				\item\label{ALG2_STEP_2a} Let
					$$
						t_j = \begin{cases}	
								t_{j+1},& \mbox{if }w_j = w_{j+1},\\
								1-t_{j+1},& \mbox{if }w_j \neq w_{j+1}.
							\end{cases}
					$$
					
				\item \label{ALG2_STEP_2b} Let $\QQQ_{t_j}^{(j)}$ be 
					$$
						\star\big(Q_{t_j,1}^{(j+1)}, \dots ,Q_{t_j,m' - 1}^{(j+1)},
							Q_{t^{(j)},m'}^{(j+1)} \setminus \{w_j\}, \{w_j\}\big),
					$$
					where $m' = m_{t_j}^{(j+1)}$.

				\item\label{ALG2_STEP_2c} Fix notation
					$$
						m := m^{(j+1)}_{1-t_j}\mbox{ and }
						\QQQ^{(j+1)}_{1-t_j} = \big(Q_1,\dots, Q_m\big).
					$$

				\item\label{ALG2_STEP_2d} If for some $i\in \INTi{1,m}$ we have $\LLL_j \subseteq Q_i$
					and  $w_j \in Q_i$,
					let $\QQQ_{1-t_j}^{(j)}$ be
						$$
							 \star\big(Q_1,\dots, Q_{i-1},Q_i \setminus \LLL_j
							, Q_{i+1},\dots, Q_m, \LLL_j\big).
						$$
				\item\label{ALG2_STEP_2e} If for some $i\in \INTi{1,m}$ we have $\LLL_j \subseteq Q_i$
					but $w_j \not\in Q_i$,
					let $\QQQ_{1-t_j}^{(j)}$ be
						$$
							 \star\big(Q_1,\dots, Q_{i-1}\setminus \{w_j\}, \{w_j\},Q_i \setminus \LLL_j
							, Q_{i+1},\dots, Q_m, \LLL_j\big).
						$$
						
				\item\label{ALG2_STEP_2f} If there exist $i_0,i_i\in \INTi{1,m}$, $i_0 <i_1$, so that
					\begin{equation}\label{EQ_ALG_PAIR_COND}
						\LLL_j \cap Q_i \neq \emptyset \mbox{ for all } i\in \INTi{i_0,i_1}
						\mbox{ and }
						\LLL_j \subseteq \bigcup_{i=i_0}^{i_1} Q_i
					\end{equation}
					and $w_j \in Q_{i_0}$, then let $\QQQ_{1-t_j}^{(j)}$ be
					$$
						\begin{array}{lr}
						\multicolumn{2}{l}{\star\big(Q_1,\dots,Q_{i_0} \setminus (\LLL_j\cup\{w_j\}),\{w_j\}, Q_{i_1}\setminus \LLL_j,
							Q_{i_1+1},\dots, Q_m, \quad\quad\quad}\\
							& Q_{i_0}\cap \LLL_j, Q_{i_0+1},\dots, Q_{i_1-1},Q_{i_1}\cap \LLL_j\big).
						\end{array}
					$$
				\item\label{ALG2_STEP_2g} If \eqref{EQ_ALG_PAIR_COND} holds but $w_j \not\in Q_{i_0}$,
					then let $\QQQ_{1-t_j}^{(j)}$ be
					$$
						\begin{array}{lr}
						\multicolumn{2}{l}{ \star\big(Q_1,\dots, Q_{i_0-1}\setminus \{w_j\},\{w_j\},Q_{i_1} \setminus \LLL_j,
							Q_{i_1+1},\dots, Q_m, \quad\quad\quad}\\
							& Q_{i_0}, Q_{i_0+1},\dots, Q_{i_1-1},Q_{i_1}\cap \LLL_j\big).
						\end{array}
					$$
			\end{enumerate}
		
	\end{enumerate}
	\end{algo}
	
	\subsection{Examples}
	
	Before stating and proving our main result,
		we will for concreteness demonstrate Algorithm \ref{ALGO_ZORICH}
		on a few example Zorich paths.
	We will also make a few observations about the outcomes.
	
	\begin{exam}
		For $\AAA = \INTi{1,6}$, consider the Zorich path given by winner/loser pairs
			$$
				(w_1,\LLL_1) = \big(1,\{2,3\}\big),~
				(w_2,\LLL_2) = \big(4,\{1,5\}\big),~
				(w_3,\LLL_3) = \big(6,\{2,3,4\}\big),
			$$
		and we apply Algorithm \ref{ALGO_ZORICH}.
		Starting with Step \stepref{ALG2_STEP_1}, we make $t_3 = 0$ and
			$$
				\big(\QQQ_0^{(3)},\QQQ_1^{(3)}\big) = \big(\big(\{1,2,3,4,5\},\{6\}),(\{1,5,6\},\{2,3,4\}\big)\big).
			$$
		We then move to $j=2$.
		Because $w_2 \neq w_3$, we have $t_2 = 1$ (Step \stepref{ALG2_STEP_2a}) and by Step \stepref{ALG2_STEP_2b}
			we get
				$$
					\QQQ_1^{(2)} = \big(\{1,5,6\},\{2,3\},\{4\}\big),
				$$
			and by Step \stepref{ALG2_STEP_2d} we get
				$$
					\QQQ_0^{(2)} = \big(\{2,3,4\},\{6\},\{1,5\}\big).
				$$
		We then consider $j=1$.
		Following Step \stepref{ALG2_STEP_2a} we have $t_1 = 0$.
		Step \stepref{ALG2_STEP_2b} then gives
			$$
				\QQQ_0^{(1)} = \big(\{2,3,4\},\{6\},\{5\},\{1\}\big),
			$$
		and Step \stepref{ALG2_STEP_2g} gives
			$$
				\QQQ_1^{(1)} = \big(\{5,6\},\{1\},\{4\},\{2,3\}\big).
			$$
		We then have that the possible pairs in $\PAIRS{\AAA}$ that can start
			this Zorich path must satisfy the following:
				one map $p_t$ places letter $6$ in position $4$, letter $5$ in position $5$
					and letter $1$ in position $6$
				and the other map $p_{1-t}$ places letters $5$ and $6$ as the first two letters,
					places letters $2$ and $3$ as the last two letters,
						places letter $1$ in position $3$ and places letter $4$ in position $4$.
	\end{exam}
	
	In the previous example, we were left with uncertainty concerning the pair that started this Zorich path.
	In fact, the reader may verify that any pair satisfying those conditions must be irreducible, so there are
			$2 \cdot 2! \cdot 2! \cdot 3! = 48$ possible starting irreducible pairs.
	This is ambiguity is not surprising as the path was not complete (letters $2,3,5$ did not win).
	The next example will consider a $1$-complete path, and the outcome will be quite different.

	\begin{exam}
		For $\AAA = \{A,B,C,D,E\}$, consider the $1$-complete Zorich path given by winner/loser pairs
			$$
				(w_1,\LLL_1) = \big(E,\{A,B\}\big),~
				(w_2,\LLL_2) = \big(C,\{E\}\big),~
				(w_3,\LLL_3) = \big(D,\{C\}\big),
			$$
			$$
				(w_4,\LLL_4) = \big(C,\{D\}\big),~
				(w_5,\LLL_5) = \big(E,\{C,D\}\big),
			$$
			$$
				(w_6,\LLL_6) = \big(A,\{C,D,E\}\big),~
				(w_7,\LLL_7) = \big(B,\{A\}\big),
			$$
		By Algorithm \ref{ALGO_ZORICH}:
			$$
				\begin{array}{cl}
					\QQQ_0^{(7)} = \big(\{A,C,D,E\},\{B\}\big) \mbox{ and }
					\QQQ_1^{(7)} = \big(\{B,C,D,E\},\{A\}\big),&   \mbox{by Step }\stepref{ALG2_STEP_1},\\
					\QQQ_1^{(6)} = \QQQ_1^{(7)} = \big(\{B,C,D,E\},\{A\}\big),&   \mbox{by Step }\stepref{ALG2_STEP_2b},\\
					\QQQ_0^{(6)} = \big(\{A\},\{B\},\{C,D,E\}\big), &  \mbox{by Step }\stepref{ALG2_STEP_2d},\\
					\QQQ_0^{(5)} = \big(\{A\},\{B\},\{C,D\},\{E\}\big), &  \mbox{by Step }\stepref{ALG2_STEP_2b},\\
					\QQQ_1^{(5)} = \big(\{B,E\},\{A\},\{C,D\}\big),&   \mbox{by Step }\stepref{ALG2_STEP_2d},\\
					\QQQ_1^{(4)} = \big(\{B,E\},\{A\},\{D\},\{C\}\big),&   \mbox{by Step }\stepref{ALG2_STEP_2b},\\
					\QQQ_0^{(4)} = \big(\{A\},\{B\},\{C\},\{E\},\{D\}\big), &  \mbox{by Step }\stepref{ALG2_STEP_2d},\\
					\QQQ_0^{(3)} =\QQQ_0^{(4)} = \big(\{A\},\{B\},\{C\},\{E\},\{D\}\big), &  \mbox{by Step }\stepref{ALG2_STEP_2b},\\
					\QQQ_1^{(3)} = \QQQ_1^{(4)} = \big(\{B,E\},\{A\},\{D\},\{C\}\big),&   \mbox{by Step }\stepref{ALG2_STEP_2e},\\
					\QQQ_1^{(2)} = \QQQ_1^{(3)} = \big(\{B,E\},\{A\},\{D\},\{C\}\big),&   \mbox{by Step }\stepref{ALG2_STEP_2b},\\
					\QQQ_0^{(2)}  = \big(\{A\},\{B\},\{C\},\{D\},\{E\}\big), &  \mbox{by Step }\stepref{ALG2_STEP_2e},\\
					\QQQ_0^{(1)}  = \QQQ_0^{(2)}  = \big(\{A\},\{B\},\{C\},\{D\},\{E\}\big), &  \mbox{by Step }\stepref{ALG2_STEP_2b},\\
					\QQQ_1^{(1)} = \big(\{E\},\{D\},\{C\},\{B\},\{A\}\big),&   \mbox{by Step }\stepref{ALG2_STEP_2f},\\
				\end{array}
			$$
		and we conclude that the only elements of $\PAIRS{\AAA}$ that can begin this Zorich path is
			$(p_0,p_1)$ and its inverse $(p_1,p_0)$ where
				$p_0$ orders $\AAA$ alphabetically and $p_1$ reverses the order.
	\end{exam}
	
	As opposed to the first example, there is a unique starting pair up to taking inverses.
	However, as our next example will demonstrate, it not generally true that simply being complete is enough to
		ensure uniqueness.
	
	\begin{exam}
		For $\AAA = \INTi{1,8}$, consider the $1$-complete Zorich path given by winner/loser pairs
			$$
				(w_1,\LLL_1) = \big(8,\{1,2,3,4,6\}\big),~
				(w_2,\LLL_2) = \big(7,\{8\}\big),~
				(w_3,\LLL_3) = \big(6,\{7\}\big),
			$$
			$$
				(w_4,\LLL_4) = \big(5,\{6\}\big),~
				(w_5,\LLL_5) = \big(4,\{5\}\big),~
				(w_6,\LLL_6) = \big(3,\{4\}\big),
			$$
			$$
				(w_7,\LLL_7) = \big(2,\{3\}\big),~
				(w_8,\LLL_8) = \big(1,\{2\}\big).
			$$
		By Step \stepref{ALG2_STEP_1} followed by repeated applications of Step \stepref{ALG2_STEP_2d}
			we get
			$$
				\QQQ_0^{(2)} = \big(\{2,4,6,8\},\{1\},\{3\},\{5\},\{7\}\big),
			$$
			$$
				\QQQ_1^{(2)} = \big(\{1,3,5,7\},\{2\},\{4\},\{6\},\{8\}\big).
			$$
		We then apply Step \stepref{ALG2_STEP_2f} to get
			$$
				\QQQ_0^{(1)} = \big(\{2,8\},\{5\},\{7\},\{4,6\},\{1\},\{3\}\big),
			$$
			$$				
				\QQQ_1^{(1)} = \QQQ_1^{(2)} = \big(\{1,3,5,7\},\{2\},\{4\},\{6\},\{8\}\big).
			$$
		So any pair $(p_0,p_1)\in \PAIRS{\AAA}$ or its inverse $(p_1,p_0)$ that agrees with
			$(\QQQ_0^{(1)}, \QQQ_1^{(1)})$ can be an initial pair for this Zorich path.
		There are $2\cdot 2! \cdot 2! \cdot 4! = 192$ possibilities
			as we may verify that any choice will be irreducible.
	\end{exam}
	
	Before moving to the next section, we make a few observations that form the primary arguments in the next section's proof.
	First, because every letter wins in a complete path, the application of Steps \stepref{ALG2_STEP_1}/\stepref{ALG2_STEP_2b}
		ensure that each letter must appear as a singleton in at least one of $\QQQ^{(1)}_0$
			or $\QQQ^{(1)}_1$.
	Second, Steps \stepref{ALG2_STEP_2d} and \stepref{ALG2_STEP_2f} necessarily
		make $\QQQ_{1-{t_j}}^{(j)}$ with more elements than $\QQQ_{1- t_j}^{(j+1)}$.
	Finally, any of the Steps \stepref{ALG2_STEP_2e} -- \stepref{ALG2_STEP_2g}
		result in $\QQQ_{1-{t_j}}^{(j)}$ that contains $\{w_j\}$ as an element, making $\{w_j\}$
			an element of both $\QQQ_0^{(j)}$ and $\QQQ_1^{(j)}$.

	\subsection{Main Theorem and Proof}\label{SSEC_MAIN1_PROOF}
	
	Before stating our main result, we provide the following correctness result for Algorithm \ref{ALGO_ZORICH}.
	The proof follows naturally from the discussion about the algorithm's construction
		and induction on the length of a Zorich path.
	
	\begin{lemm}\label{LEM_ALGO_CORRECT}
		For a finite Zorich path on $\AAA$, let $(\QQQ_0^{(1)},\QQQ_1^{(1)})$ be the partially ordered
			pair obtained by Algorithm \ref{ALGO_ZORICH}.
		Then for any $(p_0,p_1)\in \PAIRS{\AAA}$
			the following are equivalent:
			\begin{enumerate}
				\item $(p_0,p_1)$ is an initial pair for the Zorich path,
				\item either $(p_0,p_1)$ agrees with $(\QQQ_0^{(1)}, \QQQ_1^{(1)})$
					or $(p_1,p_0)$ agrees with $(\QQQ_0^{(1)}, \QQQ_1^{(1)})$.
			\end{enumerate}
	\end{lemm}

	\begin{theo}\label{THM_ZPAIR}
		A $C$-complete Zorich path on $\AAA$, $C \geq \log_2(|\AAA|+1)-1$, has its initial pair uniquely determined up to
			inverses.
	\end{theo}
	
	\begin{proof}
		Let $n = |\AAA|$.
		Because $|\PAIRS{\AAA}| = 2$ if $n=2$, we assume $n\geq 3$.
		
		We apply Algorithm \ref{ALGO_ZORICH} to
			the winner/loser information
		$$
			(w_1,\LLL_1),(w_2,\LLL_2),\dots,(w_N,\LLL_N),
		$$
		to construct partially ordered pairs
		 $(\QQQ_0^{(j)},\QQQ_1^{(j)})$ for $t\in \{0,1\}$.
		We also define
		$$
			\QQQ_0^{(N+1)} = \QQQ_1^{(N+1)} = \{\AAA\},
		$$
		as is consistent with the discussion before Algorithm \ref{ALGO_ZORICH}.
		 
		We will show that
		\begin{equation}\label{EQ_ZPROOF_GOAL}
			\big|\QQQ_t^{(1)}\big| > n-1 \mbox{ for }t\in \{0,1\},
		\end{equation}
		as these are integer values and so this is equivalent to
			showing that $\QQQ_0^{(1)}$ and $\QQQ_1^{(1)}$
			are both ordered partitions of $\AAA$ into singletons.
		This implies the result of this theorem by Lemma \ref{LEM_ALGO_CORRECT}.
		Let
		$$
			0 =: N_0 < N_1 <\dots < N_{C-1} < N_C := N,
		$$
		be so that for each $k\in \INTi{1,C}$ the Zorich subpath
		defined by the moves $N_{k-1}+1\leq j\leq  N_{k}$
			is complete.
			
		Let
		$$
			u_t(k) := n - \big| \QQQ_t^{(N_{k-1} + 1)}\big| \mbox{ for } k\in \INTi{1,C+1},~t\in \{0,1\},
		$$
		denote the uncertainty in row $t$ once we have applied the algorithm to the $k^{th}$ complete subpath,
			with $u_t(C+1) = n-1$ representing the initial uncertainty.
		We now have that \eqref{EQ_ZPROOF_GOAL} is equivalent to
		\begin{equation}\label{EQ_ZPROOF_GOAL2}
			u_t(1) < 1 \mbox{ for }t\in \{0,1\}.
		\end{equation}
		
		Let
		$$
			\NNN_t(k) := \{a\in \AAA:~\{a\}\not\in \QQQ_t^{(N_{k-1}+1)}\},
		$$
		denote the elements of $\AAA$ that are not singletons in row $t$ once we have applied the algorithm
			to the $k^{th}$ complete subpath.
		In other words, if $a\in \NNN_t(k)$ we do not yet know its position in row $t$.
		
		If for some $t\in \{0,1\}$ and $j\in \INTi{1,N}$ we have that
			$\QQQ_t^{(j)}$ is not a partition of $\AAA$ into singletons,
			or equivalently $\big|\QQQ_t^{(j)}\big| \leq n-1$, then
		$$
			\big|\QQQ_t^{(j)}\big| \geq 1 + \big|\{a\in \AAA:~\{a\}\in \QQQ_t^{(j)}\}\big|
		$$
		as $\QQQ_t^{(j)}$ must contain at least one element that is not a singelton.
		We then have that
			\begin{equation}\label{EQ_ZPROOF_Nt_ut}
				u_t(k) + 1\leq  \big|\NNN_t(k)\big| \mbox{ if } u_t(k) \geq 1,
			\end{equation}
		by using our definitions above.
				
		For any $k\in \INTi{1,C}$ consider $a\in \NNN_r(k)$ for fixed type $r\in \{0,1\}$.
		By the completeness of our subpaths, there exists
			$j\in \INTi{N_{k-1}+1,N_k}$ such that the winner $w_j$
			is $a$.
		Because $\{a\}\not\in \QQQ_r^{(N_{k-1}+1)}$, it must be that
			$\{a\}\not\in \QQQ_r^{(j)}$ as we only refine (and reorder) our partitions in the algorithm construction.
		The following then must be true about the application of Step \stepref{ALG2_STEP_2}:
			the type $t_j$ must be $1-r$ and Step \stepref{ALG2_STEP_2d}
			must have been used, as each of Steps \stepref{ALG2_STEP_2e} -- \stepref{ALG2_STEP_2g}
				result in $\{a\}\in \QQQ_r^{(j)}$.
		In Step \stepref{ALG2_STEP_2d}, we must have that
			$$
				\big|\QQQ_r^{(j+1)}\big| = \big|\QQQ_r^{(j)}\big|+1
			$$
		as exactly one element $Q_i$ of $\QQQ_r^{(j+1)}$ contained $\LLL_j$ and $a$,
		and so its division into $Q_i \setminus \LLL_j$ and $\LLL_j$ produced two nonempty sets
		as $a\in Q_i \setminus \LLL_j$.

		Because this applies to each element of $\NNN_r(k)$ and $r\in \{0,1\}$, we have
			$$
				u_t(k) \leq u_t(k+1) - \big|\NNN_t(k)\big| \mbox{ for }k\in \INTi{1,C},~t\in \{0,1\},
			$$
		and by combining this with \eqref{EQ_ZPROOF_Nt_ut} we have
			\begin{equation}\label{EQ_ZPROOF_Nt_k+1}
				u_t(k) \leq \frac{u_t(k+1)-1}{2} \mbox{ for }k\in \INTi{1,C}, t\in \{0,1\},
			\end{equation}
		which by induction and the definition at $C+1$ gives
			\begin{equation}
				u_t(1) \leq \frac{n}{2^C} - 1 \mbox{ for }t\in\{0,1\}.
			\end{equation}
		We then satisfy \eqref{EQ_ZPROOF_GOAL2} for the integer $u_t(1)$ if
			$
				n \leq 2^{C+1} -1
			$,
		or $C \geq \log_2(n+1) - 1$, as claimed.
	\end{proof}
	
	\begin{coro}\label{COR_ZPERM}
		A $C$-complete Zorich path on $\IRR{n}$, $C \geq \log_2(n+1)-1$, has its initial permutation uniquely determined.
	\end{coro}
	
	\begin{proof}
		Given a $C$-complete Zorich path $\widetilde{A}_1,\dots, \widetilde{A}_N$ we may apply Corollary \ref{COR_LIFTING_ZOR}
			to lift to $C$-complete
			Zorich path $\widetilde{\Theta}_1,\dots, \widetilde{\Theta}_N$
			over $\AAA = \INTi{1,n}$ with natural maps $\tau_i$ such that
			$$
				\Theta_i = \Psi_{\tau_i} A_i \Psi^*_{\tau_{i+1}} \mbox{ for } i\in \INTi{1,N}.
			$$
		If for some $i\in \INTi{1,N}$ we have at least one entry with value at least $2$ in $\widetilde{\Theta}_i$,
			we divide according to Section \ref{SSEC_BREAKUP}.
		We may then apply Theorem \ref{THM_ZPAIR} to get the inital pair
			$(p_0,p_1)$ of the Zorich path over $\AAA$ up to inverses.
		By Corollary \ref{COR_LIFTING_ZOR}, exactly one of $\Pi(p_0,p_1)$, $\Pi(p_1,p_0)$
			may start the Zorich path in $\IRR{n}$.
	\end{proof}
	
	\section{Sharpness of Theorem \ref{THM_ZPAIR}}\label{SEC_MAIN2}

In this section, we will address the sharpness of our main result.
We first note that $C_{sharp}= \lceil\log_2(|\AAA|+1) - 1\rceil$ is the minimum $C$ such that
	a $C$-complete Zorich path is guaranteed a unique initial pair
	by Theorem \ref{THM_ZPAIR}.
If $|\AAA| = 3$, then Algorithm \ref{ALGO_ZORICH} will determine our
	initial pair (up to inverses) after any $1$-complete path.
If $|\AAA|\in \INTi{4,7}$,
	a $1$-complete path may not suffice but we are guaranteed a unique
	initial pair given a $2$-complete path.
	
For a given $|\AAA| \geq 8$ we will construct a $C_{flat}$-complete
	Rauzy path such that its initial pair cannot be determined (even up to inverses),
	where $C_{flat} = \lfloor \log_2(|\AAA|) -  1\rfloor$.
The construction of such a path will be the main content of this section and we will prove the following.

\begin{theo}\label{THM_NEED_C}
		For $\AAA$ such that $|\AAA| \geq 8$, there exists
		a Rauzy path on $\AAA$
		that is $\lfloor \log_2(|\AAA|)  - 1 \rfloor$-complete and
	the initial pair cannot be determined.

\end{theo}



Because each loser set $\LLL_j$ is a singleton for a step of Rauzy induction,
	we will instead refer to a Rauzy path by the winner loser pairs
	$$
		(w_1,\ell_1),\dots, (w_N,\ell_N)
	$$
where $\ell_j$ is defined by $\LLL_j = \{\ell_j\}$ for $j\in \INTi{1,N}$.
It follow in this case that
	Steps \stepref{ALG2_STEP_2f} and \stepref{ALG2_STEP_2g} of  Algorithm \ref{ALGO_ZORICH} never occur.
We may verify then by induction that, as we apply the algorithm the partially ordered pairs
	$(\QQQ_0^{(j)},\QQQ_1^{(j)})$ must satisfy
	\begin{equation}\label{EQ_RAUZY_Q_FORM}
		\big|Q_{t,i}^{(j)}\big| \geq 2 \Rightarrow i = 1 \mbox{ for all }i\in \INTi{1,m_t^{(j)}},j\in \INTi{1,N}, t\in \{0,1\},
	\end{equation}
or, in words, the only non-singleton element of each $\QQQ_t^{(j)}$, if one exists, must be the leftmost element.

In our construction, we will create partially ordered pairs that have a very specific form.
We may consider this a partially ordered analogue of \emph{standard pairs},
	meaning pairs $(p_0,p_1)$ that satisfy $p_0(1) = p_1(|\AAA|)$ and $p_1(1) = p_0(|\AAA|)$.
These are equivalent conditions if and only if $|\QQQ_0| = |\QQQ_1| = |\AAA|$, or
	$(\QQQ_0,\QQQ_1)$ actually represents a unique pair $(p_0,p_1)\in \PAIRS{\AAA}$.
	
\begin{defi}\label{DEF_XFORM}
	For partially ordered pair $(\QQQ_0,\QQQ_1)$ on $\AAA$.
	Then $(\QQQ_0,\QQQ_1)$ is \emph{Form $X$} with pivot letters $(a_0,a_1)$ if \eqref{EQ_RAUZY_Q_FORM}
		holds\footnote{Meaning there is at most one non-singleton is each $\QQQ_t$ and it must be the leftmost element if it exists.}
		and for each $t\in \{0,1\}$ the leftmost singleton of $\QQQ_t$ is $\{a_t\}$ which is also the
		rightmost singelton of $\QQQ_{1-t}$.
\end{defi}

We will construct a Rauzy path by applying complete subpaths each starting with an irreducible partially ordered pair
	of Form $X$.
As we proceed to our main construction and proof,
	we want to show that unless our Form $X$ partially ordered pair
	is decomposed into singletons,
	it is not the case that the partially ordered pair
		represents exactly one $(p_0,p_1)\in \PAIRS{\AAA}$.
To so, we make this notion of ambiguity precise in the following definition
	and lemma.

\begin{defi}
	For irreducible partially ordered pair $(\QQQ_0,\QQQ_1)$,
		let
		\begin{equation}
			\SSS_t := \{\alpha\in \AAA:~\{a\}\in \QQQ_t\}
		\end{equation}
		denote the set of singleton letters in $\QQQ_t$.
	We say for fixed $r\in \{0,1\}$ that $\beta \in \AAA\setminus \SSS_r$
		has \emph{definitive} position $i$ in $\QQQ_r$ if for each pair
			$(p_0,p_1)\in \PAIRS{\AAA}$ that agrees with $(\QQQ_0,\QQQ_1)$
		we must have $p_r(\beta) = i$.
\end{defi}

\begin{lemm}\label{LEM_ARE_NOT_DONE}
	For irreducible Form $X$ partially ordered pair $(\QQQ_0,\QQQ_1)$
		and for each $t\in \{0,1\}$,
		no letter in $\AAA\setminus \SSS_t$ has a definitive position in $\QQQ_t$.
\end{lemm}

\begin{proof}
	Let $n = |\AAA|$ and $m_t = \big|\QQQ_t\big|$ for $t\in \{0,1\}$.
	We may without loss of generality assume
		$$
			m_0 \leq m_1 \leq n,
		$$
	and will argue by cases on $m_0,m_1$.
	Note that either $m_t = n$ or
		$$
			m_t = 1 + \big|\SSS_t\big| \mbox{ if }m_t<n.
		$$
	Recall the pivot letters $a_0,a_1$ from Definition \ref{DEF_XFORM}.
	
	First if $m_0 = n$, then the statement is trivially true as $\SSS_0 = \SSS_1 = \AAA$.
	Now, assume $m_1 = n$ but $m_0 \leq n-1$.
	If we choose any pair $(p_0,p_1)$ that agrees with $(\QQQ_0,\QQQ_1)$ then it must
		be irreducible as $p_0(a_0) = 1$ and  $p_1(a_0) = n$.
	Because $p_1$ may order the elements of $\AAA\setminus \SSS_1$
		with complete freedom in positions $1$ to $n - \big|\SSS_1\big| \geq 2$,
		no letter in $\AAA\setminus \SSS_1$ has a definitive position in $\QQQ_1$.
		
	If $m_1 \leq n-1$, then because $(\QQQ_0,\QQQ_1)$ is irreducible there exists
		$b_0 \in \SSS_0 \setminus \SSS_1$.
	(Note that $\big|\AAA\setminus \SSS_t\big| \geq 2$ for each $t\in \{0,1\}$.)
	We note that any pair $(p_0,p_1)$ that agrees with $(\QQQ_0,\QQQ_1)$
		and satisfies $p_1(b_0) = 1$ is irreducible,
		and so no letter in $\AAA\setminus \SSS_0$ has a definitive position
		in $\QQQ_0$.

	We now want to show that no letter in $\AAA\setminus \SSS_1$ has a definitive position
		in $\QQQ_1$.		
	If there exists $b_1\in \SSS_1\setminus \SSS_0$ then by a similar argument we
		must have that no letter in $\AAA\setminus \SSS_1$ has a definitive position
		in $\QQQ_1$.
	If instead $\SSS_1 \subsetneq \SSS_0$ but there exists $b_0'\neq b_0$ in $\SSS_0\setminus \SSS_1$ then we may also conclude
		that no letter in $\AAA\setminus \SSS_1$ has definitive position in $\QQQ_1$
		as a pair $(p_0,p_1)$ satisfying either $p_1(b_0) = 1$ or $p_1(b_0') = 1$ is irreducible.
	Finally, if $\SSS_0 = \SSS_1 \cup \{b_0\}$ then consider any ordering of the letters
		in $\AAA\setminus \SSS_1$ such that $b_0$ is not the rightmost in this set
		and let $p_1$ order $\AAA\setminus \SSS_1$ in this way and also agree with $\QQQ_1$.
	We may verify that if $p_0$ agrees with $\QQQ_0$ and $p_0(b_1) = 1$ where $b_1 \neq b_0$ is the
		rightmost element according to $p_1$, then $(p_0,p_1)$ is irreducible.
	Because $\big|\AAA\setminus \SSS_1\big| \geq 3$ in this case,
		letter in $\AAA\setminus \SSS_1$ has definitive position in $\QQQ_1$.
		
	By exhausting all cases, we have concluded that no Form $X$ irreducible
		partially ordered pair has such a letter in definitive position, as claimed.
\end{proof}

Our iterative path constructions will be provided implicitly within the following proof.
In it, we describe how to create Rauzy paths starting and ending at Form $X$
	partially ordered pairs that visit only specific letters
	while minimizing the addition of information.

\begin{lemm}\label{LEM_CUTTING_UNKNOWNS}
	For irreducible partially ordered pair $(\QQQ_0',\QQQ_1')$ on $\AAA$ of Form $X$,
		let $\SSS_t'$ denote the set of singletons in $\QQQ_t'$.
	If $\SSS_0' \cup \SSS_1' = \AAA$ then:
	\begin{enumerate}
		\item there exists a Rauzy path beginning and ending at $(\QQQ'_0,\QQQ'_1)$
			such that each letter in $\SSS'_0 \cap \SSS'_1$ wins and
				no other letters win, and
		\item For each $r\in \{0,1\}$,
			there exists a Rauzy path ending at $(\QQQ_0',\QQQ_1')$
			and starting at a partially ordered pair $(\QQQ_0,\QQQ_1)$
			of Form $X$ such that: each letter in $\SSS'_r \setminus \SSS'_{1-r}$ wins at least once,
				$\SSS_{1-r}\setminus \SSS_r = \SSS_{1-r}' \setminus \SSS'_r$
				and
					$$\big|\SSS_r \setminus \SSS_{1-r}\big| = \begin{cases}
						\left\lfloor \big|\SSS_r' \setminus \SSS_{1-r}'\big|/2\right\rfloor, & \mbox{if }\big|\SSS_r' \setminus \SSS_{1-r}'\big| \geq 4\\
						0, & \mbox{otherwise.}
						\end{cases}
				$$
	\end{enumerate}
\end{lemm}

\begin{proof}
	For the first item,
	let $b_1,\dots,b_m$ be the elements of $\SSS_0'\cap \SSS_1'$ ordered in reverse with respect to $\QQQ_0'$,
		meaning $\{b_i\}$ is to the right of $\{b_{i+1}\}$ for each $i\in \INTi{1,m-1}$.
	Note that $b_1$ is the pivot letter $a_1$ and $b_m$ is the pivot letter $a_0$.
	
	For each $i\in \INTi{1,m}$ let $h_i$  and $j_i$ be the positions
		of $\{b_i\}$ according to $\QQQ_0'$ and $\QQQ_1'$ respectively.
	Because $(\QQQ_0',\QQQ_1')$ is Form $X$,
	$$
		h_m = n - |\SSS_0| +1,~h_1 = n,~j_1 = n - |\SSS_1| + 1,\mbox{ and }
			j_m = n,
	$$
	 and because $(\QQQ_0',\QQQ_1')$ is also irreducible $m\geq 2$.
	The Rauzy path is described in order by type from $(\QQQ_0',\QQQ_1')$ as
		$$
			0^{n - j_1} 1^{h_1 - h_2} 0^{n- j_2}1^{h_2 - h_3} 0^{n- j_3} \dots 0^{n - j_{m-1}}1^{h_{m-1} - h_m},
		$$
	meaning we first apply $n-j_1$ moves of type $0$, followed by $h_1-h_2$ moves of type $1$, ets.
	We may verify that:
		\begin{enumerate}
			\item each move of type $1$ has $b_m$ as the winner and belongs to a cycle,
			\item the moves in the $0^{n-j_i}$ cycle has $b_i$ as the winner and
			\item this path returns to $(\QQQ_0',\QQQ_1')$.
		\end{enumerate}
		
	For the second item, consider the pivot letters $a_0',a_1'$ for Form $X$ $(\QQQ_0',\QQQ'_1)$.
	Note that $a_0',a_1'\in \SSS_0\cap\SSS_1$ and $a_{1-t}'$ is the rightmost element of $\QQQ'_t$
		for $t\in \{0,1\}$.
	To prove our claim, we will instead show to how to take two distinct elements $b_1,b_2$ of
		$\SSS_r'\setminus \SSS_{1-r}'$, where $b_1$ appears to the left of $b_2$ in $\QQQ_r$,
			and create a Rauzy path such that $b_1$ and $b_2$ both win and the path results in
				partially ordered pair $(\QQQ_0,\QQQ_1)$
			of Form $X$ (with the same pivot points)
			such that $\SSS_r\setminus \SSS_{1-r} = \SSS'_r \setminus \big(\SSS'_{1-r}\cup \{b_2\}\big)$,
				$\SSS_0\cap \SSS_1 = \SSS'_0\cap \SSS'_1$ and $\SSS_{1-r}\setminus \SSS_r = \SSS'_{1-r}\setminus \SSS'_{1-r}$.
	We assume that $\big|\SSS'_r \setminus \SSS'_{1-r}\big| \geq 3$ so that we avoid the impossible condition
		$\big|\SSS_r\setminus \SSS_{1-r}\big| = 1$.
	We may then concatenate such paths after pairing up the elements of $\SSS_r\setminus \SSS_{1-r}$ to
		achieve our result.
		
	If $\big|\SSS_r\setminus \SSS_{1-r}\big|$ is odd,
		we may apply a path that has two parts, the latter consists of moves of type $1-r$
			so that our remaining unpaired letter $b_3 \in \SSS_r\setminus \SSS_{1-r}$
			starts the second half as the rightmost letter in row $r$
		and the former consisting of a cycle of $1-r$ moves (with $b_3$ as the winner)
		with $a_{1-r}$ as the rightmost letter in row $r$ both before and after this cycle.
	The partially ordered pair that begins this path will be Form $X$, with pivot letters $b_3$ for row $1-r$ and $a'_r$ for row $r$.
	(In particular, $b_3$ is a singleton in each row of the starting partially ordered pair.)
	
	We now create the path that acts on $b_1$ and $b_2$ as described above.
	We create a path from four subpaths, and we describe them in order from fourth to first.
	The fourth subpath consists of type $1-r$ moves so that $b_1$ is the rightmost letter
		of row $r$ according to the partially ordered pair that starts this fourth subpath.
	So at the start of this subpath, the letters $a_r$, $b_2$, $a_{1-r}$, $b_1$ appear in
		relative position in row $r$ from left to right.
	(Other letters may appear in between.)
	The third subpath consists of one move,
		type $r$ with $b_1$ as the winner and loser $b_2$.
	The partially ordered pair that starts this subpath has $b_1$ as its rightmost letter in row $r$,
		$b_2$ as its rightmost letter in row $1-r$ and has $b_2$ as a singleton in each row.
	For the second subpath, we apply moves of type $r$ so that $a_{1-r}$ is the rightmost letter
		in the partially ordered pair starting this subpath.
	We note that $b_2$ is the winner of these moves.	
	We then create our first subpath buy applying moves of type $r$ so that $a_r$ is the rightmost letter of
		row $1-r$ in the starting partially ordered pair $(\QQQ_0,\QQQ_1)$.
	Note that the only move that resulted in Step \stepref{ALG2_STEP_2e} was the move with winner $b_1$ and loser $b_2$.
	All other moves resulted in Step \stepref{ALG2_STEP_2d} so that no partition elements were divided.
	It follows that $\SSS_r = \SSS_r'$ and $\SSS_{1-r} = \SSS_{1-r} \cup \{b_2\}$ as claimed.
	
	When we iterate this construction on all paired elements, we are guaranteed to always have at least three
		singletons in row $r$ that are not in row $1-r$ as long as $|\SSS_r'\setminus \SSS_{1-r}'|\geq 4$,
		which we assume in the statement of the lemma.
\end{proof}

\begin{proof}[Proof of Theorem \ref{THM_NEED_C}]
	Without loss of generality, assume $\AAA = \INTi{1,n}$.
	We will build the sequence of reverse Rauzy moves in the order encountered by Algorithm \ref{ALGO_ZORICH},
		meaning we will list segments of the form
		$$
			(w_{M_2},\ell_{M_2}),~(w_{M_2-1},\ell_{M_2-1}),~(w_{M_2 - 2},\ell_{M_2-2}),~\dots,~(w_{M_1},\ell_{M_1}),
		$$
		where $(w_j,\ell_j)$ is the winner/pair for step $j$ and $M_1 < M_2$.
	
	The first segment we include will be the winner/loser pairs
		\begin{equation}\label{EQ_SHARP_PROOF_PATH_C2}
			(1,2),~(2,3),~(3,4),~\dots,~ (n-1,n)
		\end{equation}
	By following these moves in Algorithm \ref{ALGO_ZORICH}, if $n$ is even we are at partially ordered pair
		$$
			\mtrx{ \{2,4,\dots, n\} & \{1\} & \{3\} & \dots & \{n-1\}\\
				\{1,3,\dots,n-1\} & \{2\} & \{4\} & \dots & \{n\}}
		$$
	where the singleton set for Row $0$ is the odd elements $\{1,3,\dots, n-1\}$
	and the singleton set for Row $1$ is the even elements $\{2,4,\dots, n\}$.
	In this case we then apply the segment of winner/loser pairs
		\begin{equation}\label{EQ_SHARP_PROOF_PATH_C1}
			(n,n-2),~(n,n-4),~\dots ,~ (n,4),~(n,2),
		\end{equation}
	we have a complete subpath and are now at
		$$
			(\QQQ'_0,\QQQ'_1) = \left(\begin{array}{c c c c c c c c c}
				\{n\} & \{1\} & \{3\} & \dots & \{n-1\} & \{n-2\} & \{n-4\} & \dots & \{2\}\\
				\multicolumn{4}{c}{\{1,3,\dots,n-1\}} & \{2\} & \{4\} & \dots & \dots & \{n\}
			\end{array}\right).
		$$
	For this partially ordered pair, the singleton set for Row $0$ is now $\SSS_0'=\AAA$ and
		the singleton set for Row $1$ is (still) $\SSS_1' = \{2,4,\dots, n\}$.

	If instead $n$ is odd we are at partially ordered pair
		$$
			\mtrx{ \{2,4,\dots, n-1\} & \{1\} & \{3\} & \dots & \{n-2\} & \{n\}\\
				\multicolumn{2}{c}{\{1,3,\dots,n-2,n\}} & \{2\} & \{4\} & \dots & \{n-1\}}
		$$
	after applying the winner/loser pairs \eqref{EQ_SHARP_PROOF_PATH_C2}.
	The singleton sets are again the odds for Row $0$ and the evens for Row $1$.
	We then apply the winner/loser pairs
		\begin{equation}\label{EQ_SHARP_PROOF_PATH_C1p}
			(n,n-2),~(n,n-4),~\dots ,~ (n,3),~(n,1),
		\end{equation}
	to have a complete subpath and are at partially ordered pair
		$$
			(\QQQ_0',\QQQ_1') = \mtrx{ \multicolumn{4}{c}{\{2,4,\dots, n-1\} } & \{1\} & \{3\} & \{5\} & \dots & \{n\}\\
				\{n\} & \{2\} & \{4\} & \dots & \{n-1\} & \{n-2\} & \{n-4\} & \dots  & \{1\}}			
		$$
	with singleton sets $\SSS_0'$ the odds and $\SSS_1' = \AAA$.
	
	In either of the above cases, we have described a complete subpath
		with initial partially ordered pair $(\QQQ_0',\QQQ_1')$ that satisfies \eqref{EQ_RAUZY_Q_FORM}
		as well as the Form $X$ conditions from Definition \ref{DEF_XFORM}.
	Furthermore, the singleton sets for $(\QQQ_0',\QQQ_1')$ satisfy
		\begin{equation}
			\min\{|\SSS_0'|,|\SSS_1'|\} = \left\lceil \frac{n}{2}\right\rceil
				\Rightarrow n - \min\{|\SSS_0'|,|\SSS_1'|\} = \left\lfloor \frac{n}{2}\right\rfloor
		\end{equation}
		
	We now describe the creation of the full $C$-complete path by discussing the $C-1$ remaining  complete subpaths.
	We refer to the row $r$ such that $\SSS'_r \neq \AAA$ as the \emph{unknown row}.
	We apply moves that fulfill the conditions of the first result in Lemma \ref{LEM_CUTTING_UNKNOWNS}
		to ensure the singletons common to each row win at least once and apply moves as given in the second part to
		allow the other letters to win.
	As stated in that lemma, this will result in a new initial partially ordered pair that also satisfies \eqref{EQ_RAUZY_Q_FORM}
		as well as the Form $X$ conditions from Definition \ref{DEF_XFORM}.
	Furthermore, as long as there were at least $4$ letters missing from the singleton set in the unknown row,
		then the new unknown row will have half of the unknown letters from the previous iteration.
		
	If we call $(\QQQ_0,\QQQ_1)$ the partially ordered pair obtained after running this $C$-complete path through Algorithm \ref{ALGO_ZORICH},
		then for its singleton sets $\SSS_0$, $\SSS_1$ we have
			$$
				n - \min\{|\SSS_0|, |\SSS_1|\} = \left\lfloor \frac{n - \min\{|\SSS_0'|,|\SSS_1'|\}}{2^{C-1}}\right \rfloor
					= \left\lfloor \frac{n}{2^C}\right\rfloor
			$$
	and in order to ensure $(\QQQ_0,\QQQ_1)$ is not completely determined, we require then that
		$$
			n - \min\{|\SSS_0|,|\SSS_1|\} \geq 2
		$$
	or $C \leq \log_2(n) - 1$.
	By Lemma \ref{LEM_ARE_NOT_DONE} there are multiple initial pairs, even up to inverses, and
		so by Lemma \ref{LEM_ALGO_CORRECT} all of these are initial pairs to this Rauzy path.
\end{proof}

\section{Conclusions}

In this paper, we provided an effective proof to the uniqueness result from \cite{cVeechBUF}
	and also extended to include the more general Zorich type of induction.
It would be interesting to develop an effective result and algorithm in the more general setting 
	of matrix products as addressed in \cite{cFick}.
	
Furthermore, Rauzy induction is defined for linear involutions that are similar to IETs in some respects
	but are combinatorially more involved.
We refer to \cite{cBoissyLanneau} for more a comprehensive discussion on this induction.
Is there an analogous effective method in this setting?

\appendix

\renewcommand{\theequation}{\textbf{\color{blue}\Alph{section}.\arabic{equation}}}

\section{Algorithm for Permutations}

The proof of Corollary \ref{COR_ZPERM} uses Algorithm \ref{ALGO_ZORICH} on pairs and the lifting results given in Lemma \ref{LEM_LIFTING_UP} and Corollary \ref{COR_LIFTING_ZOR}.
In this section, we translate Algorithm \ref{ALGO_ZORICH} for direct use on permutations.

As in our main algorithm section, we let $\QQQ = (Q_1,\dots, Q_m)$
	be a partial ordering on $\INTi{1,n}$, meaning
		$\{Q_i:~i\in \INTi{1,m}\}$ is a partition of $\INTi{1,n}$.
Irreducible permutation $\pi \in \IRR{n}$ agrees with partial ordering $\QQQ$
	if for each $j,j'\in \INTi{1,m}$ and $i,i'\in \INTi{1,n}$ such that
		$i \in Q_j$ and $i'\in Q_{j'}$ we have $\pi(i) < \pi(i')$.
We also reuse the $\star$ notation from \eqref{EQ_STAR}, specifically if  $\QQQ = (Q_1,\dots, Q_m)$
	is a tuple of sets, then $\star \QQQ$ is the tuple after removing any entries that are the empty set.
	
	The Zorich induction matrices $\tilde{A}$ tell us the type $t$.
	When the matrix is type $0$ the non-zero entries in row $n$ tell us the losers
		with multiplicity, and we break these up as discussed in Section \ref{SSEC_BREAKUP}
			to ensure the maximum entry is $1$ and we get the loser set
			$$
				\LLL = \{i \in \INTi{1,n-1}:~A_{n,i} = 1\}.
			$$
	When the matrix is type $1$, we know the value $k =\pi^{-1}(n)$ and the number of moves $p$ of this type in a row,
		which is also the power of the matrix \eqref{EQ_A1_PERMS} that forms this matrix product.
	
	Algorithm \ref{ALGO_ZORICH_PERM} follows Algorithm \ref{ALGO_ZORICH}
		while applying the appropriate relabeling as discussed in Section \ref{SSEC_PERM+PAIRS}.
	Because this relabeling always treats a permutation $\pi \in \IRR{n}$
		as a pair $(\ID{n},\pi) \in \PAIRS{\INTi{1,n}}$, we expect to follow Algorithm \ref{ALGO_ZORICH} when
			the move is type $0$.
			
	When the move is type $1$ we do not gain any information about Row $0$ but instead must only relabel the known
		letters within the partial ordering for $\pi$.
	To do this, we define a map $\delta_k:\INTi{1,n} \to \INTi{1,n}$ by
		$$
			\delta_k(j) = \begin{cases}
				j, & j \leq k\\
				n, & j = k+1,\\
				j-1, & j \geq k+2,
			\end{cases}
		$$
	and extend this map to set $S \subseteq \INTi{1,n}$ by
		$$
			\delta_k(S) = \{\delta_k(s):~s\in S\}.
		$$
		
\subsection{The Permutation Algorithm}

Given the discussion in the previous section, we now present the translated algorithm.
	
	\begin{algo}\label{ALGO_ZORICH_PERM}
	Assume $n \geq 3$ and consider matrices
		$$
			\tilde{A}_1 ,\dots \tilde{A}_N
		$$
	of a Zorich path of length $N\geq 1$ on $\IRR{n}$.
	\begin{enumerate}
		\renewcommand{\theenumi}{\texttt{\arabic{enumi}}}
		\renewcommand{\labelenumi}{\texttt{\color{purple}$\langle$\theenumi$\rangle$}}
		
		\item Let $t_N$ be the type of matrix $\tilde{A}_N$.
		
			\begin{enumerate}
				\renewcommand{\theenumii}{\texttt{.\arabic{enumii}}}
				\renewcommand{\labelenumii}{\texttt{\color{purple}$\langle$\theenumi\theenumii$\rangle$}}
				
				\item\label{ALG3_STEP_1a} If $t_N = 0$ then for given set $\LLL_N$,
					$$
						\QQQ^{(N)}:= \big(\INTi{1,n} \setminus \LLL_N, \LLL_n\big)
					$$
					
				\item\label{ALG3_STEP_1b} If $t_N = 1$ then for value $k_N$,
					$$
						\QQQ^{(N)} := \big(\INTi{1,n} \setminus \{k_N\}, \{k_N\}\big)
					$$
			
			\end{enumerate}
			
		\item Starting with $j = N-1$ and iterating down to $j=1$:
			\begin{enumerate}
				\renewcommand{\theenumii}{\texttt{.\arabic{enumii}}}
				\renewcommand{\labelenumii}{\texttt{\color{purple}$\langle$\theenumi\theenumii$\rangle$}}
				
				\setcounter{enumii}{-1}
				
				\item Consider previously created partial ordering, using notation
					$$\QQQ^{(j+1)} = \big(Q_1, \dots, Q_m\big).$$

				\item\label{ALG3_STEP_2a} If the type is $t_j = 1$ with additional data $k_j$ and power $p_j$, then
					$$
						\QQQ^{(j)}: = \big(\delta^{p_j}_{k_j}(Q_1),\dots, \delta^{p_j}_{k_j}(Q_{m})\big)
					$$
					and consider the next $j$.
					
				\item If the type is $t_j = 0$ let $\LLL = \LLL_j$ denote the associated set and continue to the following.

						\item\label{ALG3_STEP_2b-i} If for some $i \in \INTi{1,m}$ we have $\LLL \subseteq Q_i$ and $n \in Q_i$, then
							$$
								\QQQ^{(j)} := \star\big(Q_1,\dots, Q_{i-1}, Q_i \setminus \LLL, Q_{i+1} \dots, Q_m, \LLL \big)
							$$
							
						\item\label{ALG3_STEP_2b-ii} If for some $i \in \INTi{1,m}$ we have $\LLL \subseteq Q_i$ but $n \not\in Q_i$, then
							$$
								\QQQ^{(j)} := \star \big(Q_1,\dots,Q_{i-2}, Q_{i-1}\setminus \{n\},\{n\}, Q_i \setminus \LLL, Q_{i+1},\dots, Q_m, \LLL\big)
							$$
							
						\item\label{ALG3_STEP_2b-iii} If there exist $i_0,i_1 \in \INTi{1,m}$, $i_0 < i_1$, so that
							\begin{equation}\label{EQ_ALG_PERM_COND}
								\LLL \cap Q_i \neq \emptyset\mbox{ for all } i \in \INTi{i_0,i_1}
								\mbox{ and }
								\LLL \subseteq \bigcup_{i=1_0}^{i_1} Q_i
							\end{equation}
						and $n \in Q_{i_0}$, then
						$$
							\begin{array}{rll}
							\multicolumn{2}{l}{\QQQ^{(j)}: = \star\big(Q_1,\dots, Q_{i_0 -1},} & \quad\\
								\quad & \multicolumn{2}{r}{Q_{i_0}\setminus (\LLL\cup \{n\}), \{n\}, Q_{i_1} \setminus \LLL, Q_{i_1 + 1},\dots, Q_m,\quad}\\
								\multicolumn{3}{r}{Q_{i_0} \cap \LLL, Q_{i_0+1},\dots, Q_{i_1-1}, Q_{i_1}\cap \LLL\big).}
							\end{array}
						$$
				
						\item\label{ALG3_STEP_2b-iv} If \eqref{EQ_ALG_PERM_COND} holds but $n \not\in Q_{i_0}$, then
						$$
							\begin{array}{rll}
							\multicolumn{2}{l}{\QQQ^{(j)} := \star\big(Q_1,\dots,Q_{i_0-2},} & \quad\\
							\quad &\multicolumn{2}{r}{Q_{i_0-1} \setminus \{n\}, \{n\}, Q_{i_1} \setminus \LLL, Q_{i_1 + 1},\dots, Q_m,\quad} \\
								\multicolumn{3}{r}{Q_{i_0}, Q_{i_0+1},\dots, Q_{i_1-1}, Q_{i_1}\cap \LLL\big).}
							\end{array}
						$$
			\end{enumerate}
		
	\end{enumerate}
	\end{algo}
		
	\subsection{Permutation Examples}
	
	We end with two related example runs of Algorithm \ref{ALGO_ZORICH_PERM}.
	
	\begin{exam}
		Consider the four $5\times 5$ matrices
			$$
				\tilde{A}_1 = \mtrx{
								1 & 1 & 0 & 0 & 0\\
								0 & 0 & 1 & 0 & 0\\
								0 & 0 & 0 & 1 & 0\\
								0 & 0 & 0 & 0 & 1\\
								0 & 1 & 0 & 0 & 0},\quad
				\tilde{A}_2 = \mtrx{
								1 & 0 & 0 & 0 & 0\\
								0 & 1 & 0 & 0 & 0\\
								0 & 0 & 1 & 0 & 0\\
								0 & 0 & 0 & 1 & 0\\
								1 & 0 & 0 & 0 & 1},
			$$
			$$
				\tilde{A}_3 = \mtrx{
								1 & 1 & 1 & 0 & 0\\
								0 & 0 & 0 & 1 & 0\\
								0 & 0 & 0 & 0 & 1\\
								0 & 1 & 0 & 0 & 0\\
								0 & 0 & 1 & 0 & 0},\quad
				\tilde{A}_4 = \mtrx{
								1 & 0 & 0 & 0 & 0\\
								0 & 1 & 0 & 0 & 0\\
								0 & 0 & 1 & 0 & 0\\
								0 & 0 & 0 & 1 & 0\\
								1 & 0 & 1 & 0 & 1}.
			$$
		The fourth move is type $t_4=0$ with $\LLL_4 = \{1,3\}$, so we follow Step \newstepref{ALG3_STEP_1a} to get
			$$
				\QQQ^{(4)} = \big(\{2,4,5\}, \{1,3\}\big).
			$$
		The third move is type $t_3 = 1$ with $k_3 = 1$ and power $p_3 = 2$, so we follow Step \newstepref{ALG3_STEP_2a}
			to get
			$$
				\QQQ^{(3)} = \big(\delta_1^2(\{2,4,5\}), \delta_1^2(\{1,3\})\big)
					= \big(\{2,3,4\},\{1,5\}\big)
			$$
		The second move is type $t_2 =0$ with $\LLL_2 = \{1\}$, and we follow Step \newstepref{ALG3_STEP_2b-i}
			to get
			$$
				\QQQ^{(2)} = \big(\{2,3,4\},\{5\},\{1\}\big)
			$$
		The first move is type $t_1=1$ with $k_1 = 1$ and power $p_1$, so we follow Step \ref{ALG3_STEP_2a}
			to get
			$$
				\QQQ^{(1)} = \big(\delta_1(\{2,3,4\}), \delta_1(\{5\}), \delta_1(\{1\})\big)
					= \big(\{2,3,5\}, \{4\},\{1\}\big).
			$$
		We conclude that this Zorich path begins with $\pi\in \IRR{5}$ satisfying $\pi(1) = 5$ and $\pi(4) = 4$.
		In this case all six possible orderings of $\{2,3,5\}$ are still possible as $\pi(1)=5$ ensures any choice will
			produce an irreducible permutation.
	\end{exam}
	
	\begin{exam}
		If we have the same first four matrices followed by
			$$
				\tilde{A}_5 = \mtrx{
								1 & 1 & 1 & 1 & 0\\
								0 & 0 & 0 & 0 & 1\\
								0 & 1 & 0 & 0 & 0\\
								0 & 0 & 1 & 0 & 0\\
								0 & 0 & 0 & 1 & 0},\quad
				\tilde{A}_6 = \mtrx{
								1 & 0 & 0 & 0 & 0\\
								0 & 1 & 0 & 0 & 0\\
								0 & 0 & 1 & 0 & 0\\
								0 & 0 & 0 & 1 & 0\\
								1 & 1 & 1 & 0 & 1},
			$$
			we now get the following by Algorithm \ref{ALGO_ZORICH_PERM}:
				$$
					\begin{array}{ll}
					\QQQ^{(6)} = \{\{4,5\},\{1,2,3\}\}, &\mbox{by Step }\newstepref{ALG3_STEP_1a},\\
					\QQQ^{(5)} = \{\{2,5\}, \{1,3,4\}\}, & \mbox{by Step }\newstepref{ALG3_STEP_2a},\\
					\QQQ^{(4)} = \{\{2\},\{5\},\{4\},\{1,3\}\},  & \mbox{by Step }\newstepref{ALG3_STEP_2b-ii},\\
					\QQQ^{(3)} = \{\{4\},\{3\},\{2\},\{1,5\}\},  & \mbox{by Step }\newstepref{ALG3_STEP_2a},\\
					\QQQ^{(2)} = \{\{4\},\{3\},\{2\},\{5\},\{1\}\},  & \mbox{by Step }\newstepref{ALG3_STEP_2b-i},\\
					\QQQ^{(1)} = \{\{3\},\{2\},\{5\},\{4\},\{1\}\},  & \mbox{by Step }\newstepref{ALG3_STEP_2a},\\
					\end{array}
				$$
			and we have the unique initial permutation $\pi\in \IRR{5}$ satisfying
				$$
					\pi(1) = 5, ~\pi(2) = 2,~\pi(3) = 1,~\pi(4) = 4,~\pi(5) = 3,
				$$
			which was one of the six possibilities from the previous example.
		\end{exam}

\bibliographystyle{abbrv}
\bibliography{bibfile}

\begin{thebibliography}{10}

\bibitem{cBoissyLanneau}
C.~Boissy and E.~Lanneau.
\newblock Dynamics and geometry of the {R}auzy-{V}eech induction for quadratic
  differentials.
\newblock {\em Ergodic Theory Dynam. Systems}, 29(3):767--816, 2009.

\bibitem{cBufetov}
A.~I. Bufetov.
\newblock Limit theorems for special flows over {V}ershik transformations.
\newblock {\em Uspekhi Mat. Nauk}, 68(5(413)):3--80, 2013.

\bibitem{cBufetov+Solomyak}
A.~I. Bufetov and B.~Solomyak.
\newblock The {H}\"{o}lder property for the spectrum of translation flows in
  genus two.
\newblock {\em Israel J. Math.}, 223(1):205--259, 2018.

\bibitem{cFick}
J.~Fickenscher.
\newblock Decoding {R}auzy induction: an answer to {B}ufetov's general
  question.
\newblock {\em Bull. Soc. Math. France}, 145(4):603--621, 2017.

\bibitem{cRauzy}
G.~Rauzy.
\newblock \'{E}changes d'intervalles et transformations induites.
\newblock {\em Acta Arith.}, 34(4):315--328, 1979.

\bibitem{cVeechConj}
W.~A. Veech.
\newblock Interval exchange transformations.
\newblock {\em J. Analyse Math.}, 33:222--272, 1978.

\bibitem{cVeechBUF}
W.~A. Veech.
\newblock Decoding {R}auzy induction: {B}ufetov's question.
\newblock {\em Mosc. Math. J.}, 10(3):647--657, 663, 2010.

\bibitem{cViana}
M.~Viana.
\newblock Ergodic theory of interval exchange maps.
\newblock {\em Rev. Mat. Complut.}, 19(1):7--100, 2006.

\bibitem{cYoccoz}
J.-C. Yoccoz.
\newblock Continued fraction algorithms for interval exchange maps: an
  introduction.
\newblock In {\em Frontiers in number theory, physics, and geometry. {I}},
  pages 401--435. Springer, Berlin, 2006.

\bibitem{cZorich}
A.~Zorich.
\newblock Finite {G}auss measure on the space of interval exchange
  transformations. {L}yapunov exponents.
\newblock {\em Ann. Inst. Fourier (Grenoble)}, 46(2):325--370, 1996.

\end{thebibliography}

\end{document}